\documentclass[11pt]{article}
\usepackage{amssymb}
\usepackage{amsmath}
\usepackage{amsthm}
\usepackage{hyperref}
\usepackage{latexsym}
\usepackage{mathdots}
\usepackage{graphicx}
\usepackage[capitalize]{cleveref}
\usepackage[all]{xy}
\usepackage{xcolor}
\reversemarginpar
\theoremstyle{definition}
\newtheorem{theo}{Theorem}[section]
\newtheorem{defi}[theo]{Definition}
\newtheorem{prop}[theo]{Proposition}
\newtheorem{cor}[theo]{Corollary}
\newtheorem{lemma}[theo]{Lemma}
\newtheorem{exa}[theo]{Example}
\newtheorem{rem}[theo]{Remark}
\newtheorem{nota}[theo]{Notation}
\numberwithin{equation}{section}
\newcommand{\N}{{\mathbb N}}
\newcommand{\F}{{\mathbb F}}
\newcommand{\Z}{{\mathbb Z}}

\renewcommand{\P}{{\mathbb P}}
\newcommand{\cC}{{\mathcal C}}
\newcommand{\cD}{{\mathcal D}}
\newcommand{\cF}{{\mathcal F}}
\newcommand{\cG}{{\mathcal G}}
\newcommand{\cH}{{\mathcal H}}
\newcommand{\cM}{{\mathcal M}}

\newcommand{\cB}{{\mathcal B}}
\newcommand{\cA}{{\mathcal A}}

\newcommand{\cU}{{\mathcal U}}
\newcommand{\cL}{{\mathcal L}}
\newcommand{\cR}{{\mathcal R}}

\newcommand{\cV}{{\mathcal V}}
\newcommand{\cW}{{\mathcal W}}

\newcommand{\qM}{$q$-matroid}
\newcommand{\GL}{\mbox{\rm GL}}
\newcommand{\rk}{\mbox{\rm rk}\,}
\newcommand{\rs}{\mbox{\rm rowsp}}

\newcommand{\W}{\mbox{\rm W}}
\newcommand{\supp}{\mbox{\rm supp}}
\newcommand{\cl}{\text{cl}}

\newcommand{\wtrk}{\mbox{$\text{wt}_{\text{rk}}$}}
\newcommand{\wtHam}{\mbox{$\text{wt}_{\text{H}}$}}
\newcommand{\suppHam}{\mbox{$\text{supp}_{\text{H}}$}}
\renewcommand{\leqq}{\mbox{$\leq_{q}\, $}}
\newcommand{\leqqm}{\mbox{$\leq_{q^m}\, $}}

\newcommand{\Mat}{\mbox{\textbf{Mat}}}

\newcommand{\T}{\mbox{$\!^{\sf T}$}}

\newcommand{\subspace}[1]{\langle{#1}\rangle}

\newcommand{\Binom}[2]{\genfrac{(}{)}{0pt}{1}{#1}{#2}}
\newcommand{\BinomS}[2]{\genfrac{(}{)}{0pt}{2}{#1}{#2}}
\newcommand{\Gaussian}[2]{\genfrac{[}{]}{0pt}{1}{#1}{#2}}
\newcommand{\GaussianD}[2]{\genfrac{[}{]}{0pt}{0}{#1}{#2}}

\newenvironment{liste}{\begin{list}{--\hfill}{\topsep-1.3ex \labelwidth.4cm
   \leftmargin.5cm \labelsep.1cm \rightmargin0cm \parsep0ex \itemsep.6ex
   \partopsep-1.3ex}}{\end{list}}
\newcounter{alp}
\newcounter{ara}
\newcounter{rom}

\newenvironment{alphalist}{\begin{list}{(\alph{alp})\hfill}{\usecounter{alp}
     \topsep-1.3ex \labelwidth.7cm \leftmargin.7cm \labelsep0cm
     \rightmargin0cm \parsep0ex \itemsep0ex
     \partopsep-1.3ex}}{\end{list}}

\begin{document}
\title{Polynomial Invariants of $q$-Matroids and Rank-Metric Codes}
\author{Heide Gluesing-Luerssen\thanks{Department of Mathematics, University of Kentucky, Lexington KY 40506-0027, USA; heide.gl@uky.edu. HGL is the corresponding author.}\quad and Benjamin Jany\thanks{Department of Mathematics and Computer Science, Eindhoven University of Technology, the Netherlands; b.jany@tue.nl. BJ is supported by the Eindhoven Hendrick Casimir Institute.}
}

\date{September 24, 2025}
\maketitle
	
\begin{abstract}
\noindent
It is shown that the Whitney function of a representable $q$-matroid and the collection of all higher weight enumerators of any 
representing rank-metric code determine each other via a monomial substitution.
Moreover, the $q$-matroid itself and the collection of all higher support enumerators of the code determine each other.
Next, it is proven that the Whitney function of a $q$-matroid and the Whitney function of its projectivization determine each other via a 
monomial substitution.
 Finally, $q$-matroids with isomorphic projectivizations are studied. 
 It is shown that the projectivizations are isomorphic iff the $q$-matroids admit a dimension-preserving lattice isomorphism between 
 their lattices of flats. Such $q$-matroids are called weakly isomorphic.
\end{abstract}


\section{Introduction}\label{S-Intro}
It is well known that a Hamming-metric block code~$C$ give rise to a matroid~$M_C$ in such a way that (some of) the invariants of~$M_C$ 
are closely related to the invariants of~$C$.
The most famous relation is arguably Greene's Theorem~\cite{Gre76}, which states that the weight enumerator of~$C$ can be obtained
as an evaluation of the Witney function (or rank-generating function) of~$M_C$. 
This result has been extended in~\cite{Bri07,Bri10} (see also \cite{JuPe13}) to the higher weight enumerators.
The latter involve the weight distributions of the subcodes of~$C$.
In \cite{Bri07,Bri10} it is shown that each higher weight enumerator arises as an evaluation of the Whitney function and, conversely, the Whitney function 
equals the sum of evaluations of the higher weight enumerators.
In other words, the Whitney function of~$M_C$ and the collection of all higher weight enumerators of~$C$ carry the same information.
In~\cite{Bri10}, the author goes even further and studies not just the weight, but the actual support of codewords and subcodes.
The author proves that the collection of all higher support distributions of~$C$ and the matroid~$M_C$ fully 
determine each other in form of a matrix identity.

In this paper we generalize these results to rank-metric codes and their associated \qM{}s.
For the relation between rank-metric codes and \qM{}s we refer to \cite{JuPe18,GJLR19,Shi19,GLJ22Gen} and the next section.
We show that for a given rank-metric code~$\cC\leq\F_{q^m}^n$ and its associated \qM{}~$\cM_{\cC}$ (a)  the Whitney function of~$\cM_{\cC}$ and the collection of all higher weight enumerators of~$\cC$ determine each other, and (b) $\cM_{\cC}$ and the collection of all higher support distributions determine each other.
While these results have a striking similarity to the Hamming-metric case derived in \cite{Bri10}, there is a fundamental difference.
In the Hamming-metric case the relation between the Whitney function and the higher weight enumerators can be expressed in terms of 
 evaluations (i.e., variable substitutions of the form $(x,y)\mapsto (\alpha(x,y),\,\beta(x,y))$), and the same is true for most relations between 
polynomial invariants of the matroid, such as between the Whitney function and the Tutte polynomial or specializing the Whitney function to the characteristic polynomial.
In the rank-metric case, however, the polynomial invariants are not related via an evaluation, but rather via a monomial substitution, that is,
a $\Z$-linear map $\Z[x,y]\rightarrow\Z[x,y]$ defined by $x^iy^j\mapsto u_{i,j}$, where $u_{i,j}\in\Z[x,y]$.
More precisely, suitable monomial substitutions of the Whitney function of~$\cM_{\cC}$ result in the higher weight enumerators of~$\cC$, and conversely, 
the sum of suitable monomial substitutions of the higher weight enumerators equals the Whitney function (\cref{C-SubA}).
For $q\rightarrow1$ we recover the evaluation identities for Hamming-metric codes.
 
The central role played by monomial substitutions is also underscored by other cases.
First, the Whitney function of a \qM{} specializes to the characteristic polynomial via a monomial substitution (\cref{R-CharPoly}). 
Secondly, the convolution identity between the Whitney function 
and the Tutte polynomial of a \qM{}, recently derived in \cite{ByFu25}, can be rewritten as a monomial substitution (\cref{R-Tutte}).

The second part of this paper is devoted to the projectivization of a \qM{}. 
For a \qM{}~$\cM$ with ground space~$\F^n$, the projectivization $\P\cM$  is a matroid whose ground space is the set of lines of~$\F^n$.
It has been introduced in \cite{JPV23} and further studied in~\cite{Ja23}.
We show that the Whitney functions of~$\cM$ and $\P\cM$ are related via monomial substitutions (\cref{C-SubWhitney}).
It is known from~\cite{Ja23} that if~$\cM$ is represented by the rank-metric code~$\cC$, then $\P\cM$ is represented by the 
projectivization~$\hat{\cC}$ of~$\cC$, which is a Hamming-metric code.
We prove that the higher weight enumerators of~$\cC$ and $\hat{\cC}$ determine each other via monomial substitutions.

Examples show that non-isomorphic \qM{}s may have isomorphic projectivizations. In the last section we study this situation in detail.
It turns out that \qM{}s have isomorphic projectivizations if and only if they are weakly isomorphic, i.e., they admit a dimension-preserving 
lattice isomorphism between their lattices of flats (\cref{T-PMIsoWeaklyIso}).
We furthermore show, that for weakly isomorphic \qM{}s~$\cM_1$ and~$\cM_2$, any dimension-preserving lattice isomorphism between the flats
extends to a dimension- and rank-preserving bijection between the subspace lattices of the ground spaces of~$\cM_1$ and $\cM_2$.
This illustrates the difference to isomorphic \qM{}s, where the dimension- and rank-preserving bijection is actually induced by a vector 
space isomorphism between the ground spaces.

\begin{nota}\label{NotaIntro}
Throughout, we let $\F=\F_q$.
For any finite-dimensional $\F$-vector space~$E$ let $\cL(E)$ be the subspace lattice of~$E$.
For any field~$\hat{\F}$ and matrix $M\in\hat{\F}^{k\times n}$, we denote by $\rs(M)$ the row space of~$M$ over~$\hat{\F}$, that is, 
$\rs(M)=\{uM\mid u\in\hat{\F}^k\}$.
For $v_1,\ldots,v_t\in E$ the notation $\subspace{v_1,\ldots,v_t}$ stands for the $\F$-subspace generated by the given vectors.
We set $[n]=\{1,\ldots,n\}$ and $[n]_0=\{0,1,\ldots,n\}$.
The abbreviation NSBF stands for non-degenerate symmetric bilinear form.
Finally, for any set~$X$ we use the Kronecker-$\delta$ function, that is, for any $a,b\in X$ we set $\delta_{a,b}=1$ if $a=b$ and 
$\delta_{a,b}=0$ otherwise.
\end{nota}

\section{Basics on $q$-Matroids}\label{S-Basics}

Before starting with \qM{}s, we summarize some basic notation and terminology for classical matroids.
It will be needed throughout the paper, in particular in Sections~\ref{S-Proj} and \ref{S-WeakIso}.

\begin{rem}\label{R-ClassMatr}
The \textbf{Whitney function} of a matroid $M=(S,r)$ of rank~$K$ on the ground set~$S$ with $|S|=N$ is defined as 
\begin{equation}\label{e-WhitneyMatr}
   R_M(x,y)=\sum_{A\subseteq S}x^{K-r(A)}y^{|A|-r(A)}=\sum_{i=0}^K\sum_{j=0}^{N-K}\nu_{i,j}x^iy^j,
\end{equation}
where $\nu_{i,j}=|\{A\subseteq S\mid r(A)=K-i,\,|A|=K-i+j\}|$ (the monotonicity 
of the nullity function, $A\mapsto|A|-r(A)$, implies that $\nu_{i,j}=0$ for $j>N-K$).
Furthermore, the \textbf{characteristic polynomial} of~$M$ is
$\chi_M=\sum_{A\subseteq S}\mu(0,A)x^{K-r(A)}=\sum_{A\subseteq S}(-1)^{|A|}x^{K-r(A)}$, where $\mu$ is the M\"obius function on the subset lattice of~$S$. 
The characteristic polynomial satisfies $\chi_M(x)=(-1)^KR_M(-x,-1)$.
Two matroids $M=(S,r)$ and $M'=(S',r')$ are \textbf{isomorphic}, denoted by $M\cong M'$, if there exists a bijection $\alpha:S\longrightarrow S'$ such that $r'(\alpha(A))=r(A)$ for all subsets $A\subseteq S$.
Finally, we use the following notation for \textbf{vector matroids}. 
Let $G\in\F^{K\times N}$ be of rank~$K$.  
For any subset $\Delta\subseteq[N]$ define $G_\Delta$ as the submatrix of~$G$ consisting of the columns indexed by~$\Delta$.
The vector matroid $M_G$ is defined as $([N],r)$, where $r(\Delta)=\rk(G_{\Delta})$.
\end{rem}

We now turn to the $q$-analogues and recall the definition of \qM{}s based on rank functions.

\begin{defi}\label{D-qMatroid}
A \textbf{$q$-matroid with ground space~$E$} is a pair $\cM=(E,\rho)$, where~$E$ is a finite-dimensional $\F$-vector space
and $\rho: \cL(E)\longrightarrow\N_{\geq0}$ is a map satisfying
\begin{liste}
\item[(R1)\hfill] Boundedness: $0\leq\rho(V)\leq \dim V$  for all $V\in\cL(E)$;
\item[(R2)\hfill] Monotonicity: $V\leq W\Longrightarrow \rho(V)\leq \rho(W)$  for all $V,W\in\cL(E)$;
\item[(R3)\hfill] Submodularity: $\rho(V+W)+\rho(V\cap W)\leq \rho(V)+\rho(W)$ for all $V,W\in\cL(E)$.
\end{liste}
We call $\rho(V)$ the \textbf{rank of}~$V$ and $\rho(\cM):=\rho(E)$ the rank of the \qM.
Moreover, the \textbf{nullity} of~$V$ is defined as $\dim V-\rho(V)$.
\end{defi}

\begin{rem}\label{R-Nullity}
It is easy to see that the nullity function 
$\cL(E)\longrightarrow\N_0,\, V\longmapsto \dim V-\rho(V)$  is increasing, that is: 
$V\leq W\Longrightarrow \dim V-\rho(V)\leq \dim(W)-\rho(W)$; see also \cite[p.~99]{AlBy24}.
\end{rem}

\begin{defi}\label{D-Equiv}
Two \qM{}s $\cM_i=(E_i,\rho_i)$ are \textbf{isomorphic}, denoted by $\cM_1\cong\cM_2$, if there exists an $\F$-isomorphism $\alpha:E_1\longrightarrow E_2$ such that $\rho_2(\alpha(V))=\rho_1(V)$ for all $V\in\cL(E_1)$.
\end{defi}

A large class of $q$-matroids are the representable ones.

\begin{theo}[\mbox{\cite[Sec.~5]{JuPe18}}]\label{T-ReprqMatr}
Let $\F_{q^m}$ be a field extension of $\F$ and let $G\in\F_{q^m}^{k\times n}$.
Define the map $\rho:\cL(\F^n)\longrightarrow\N_0$ via
\[
    \rho(\rs(Y))=\rk(GY\T)\ \text{ for all matrices $Y\in\F^{t\times n}$ and all $t\in[n]_0$}
\]
(where the rank is over $\F_{q^m}$).
Then~$\rho$ is well-defined and $\cM_G:=(\F^n,\rho)$ is a $q$-matroid of rank equal to the rank of~$G$ (over $\F_{q^m}$). It is called the $q$-matroid \textbf{represented by $G$}.
A $q$-matroid $\cM$ of rank~$k$ is called \textbf{$\F_{q^m}$-representable} if $\cM\cong\cM_G$ for some matrix~$G\in\F_{q^m}^{k\times n}$.
\end{theo}

Note that $\cM_G$ only depends on the row space of~$G$, defined as $\rs(G):=\{uG\mid u\in\F_{q^m}^k\}$, and not on the actual choice of~$G$. 

The Whitney function of~$\cM$ (also known as rank-generating function or corank-nullity function) is defined in the usual way; see also  
\cite[Def.~64]{ByFu25}.

\begin{defi}\label{D-Polys}
Let $\cM=(E,\rho)$ be a \qM{} of rank~$k$ on an $n$-dimensional ground space~$E$. 
The \textbf{Whitney function} (or \textbf{corank-nullity function}) of~$\cM$ is 
\[
   R_\cM=\sum_{V\leq E}x^{k-\rho(V)}y^{\dim V -\rho(V)}\in\Z[x,y].
\]
The exponent $k-\rho(V)$ is called the \textbf{corank} of~$V$ and $\dim V -\rho(V)$ is the \textbf{nullity}.
\end{defi}

Clearly, the Whitney function is invariant under isomorphisms of \qM{}s.

One can easily describe the coefficients of the Whitney function explicitly.
A monomial $x^iy^j$ appears in~$R_\cM$ if and only if there exists a subspace $V\in\cL(E)$ such that $\rho(V)=k-i$ and 
$\dim V=k-i+j$.
In that case $j=\dim V-\rho(V)\leq\dim E-\rho(E)= n-k$ by the monotonicity of the nullity function.
As a consequence, the Whitney function is of the form
\begin{equation}\label{e-WhitExpand}
   R_{\cM}=\sum_{i=0}^{k} \sum_{j=0}^{n-k} \nu_{i,j}x^{i}y^{j}
\end{equation}
where 
\begin{equation}\label{e-WhitCoeff}
   \nu_{i,j}= \big|\{V\leq E\mid \rho(V)=k-i,\,\dim V=k-i+j\}\big|.      
\end{equation}

In \cite[Rem.~3.2]{GLJ25CF} various invariants of the \qM{} are shown to be enumerated by the Whitney function.

The following can easily be verified.

\begin{exa}\label{E-WhitneyUnif}
Let $k\in[n]_0$. The \textbf{uniform \qM{}} of rank~$k$ on an $n$-dimensional ground space~$E$, denoted by $\cU_{k,n}(q)$, is defined via the rank function $\rho(V)=\min\{k,\dim V\}$ for all $V\in\cL(E)$. 
It satisfies $R_{\cU_{k,n}(q)}=\sum_{j=0}^k\Gaussian{n}{j}x^{k-j}+\sum_{j=k+1}^n\Gaussian{n}{j}y^{j-k}$.
We call $\cU_{0,n}(q)$ and $\cU_{n,n}(q)$ the \textbf{trivial} and \textbf{free} \qM{}, respectively.
\end{exa}

The following maps will play a crucial role throughout this paper. 
They will allow us to relate various invariants of \qM{}s and rank-metric codes.

\begin{defi}\label{D-MonomSub}
For $i,j\in\N_0$ let $u_{i,j}\in\Z[x,y]$ and set $\cU=(u_{i,j})_{i,j\in\N_0}$. The $\Z$-linear map
\[
    \Omega_{\cU}:\Z[x,y]\longrightarrow\Z[x,y],\ \sum_{i,j}f_{i,j}x^iy^j\longmapsto \sum_{i,j}f_{i,j}u_{i,j}
\]
is called a \textbf{monomial substitution}.
\end{defi}

As we will show, these monomial substitutions generalize the `standard substitutions' known from matroid theory.
The following two examples provide a first illustration.

\begin{rem}\label{R-CharPoly}
The characteristic polynomial of a \qM{} $\cM=(E,\rho)$ of rank~$k$ is defined as 
$\chi_\cM=\sum_{V\leq E}\mu(0,V)x^{k-\rho(V)}\in\Z[x]$, where $\mu$ is the M{\"o}bius function on the subspace lattice $\cL(E)$ 
(see also \eqref{e-muWV} later in this paper). 
Thus, $\chi_{\cM}=\sum_{j=0}^n\sum_{\dim V=j}(-1)^jq^{\BinomS{j}{2}}x^{k-\rho(V)}$.
It is easy to verify (see also \cite[Rem.~3.6]{GLJ25CF}) that $\chi_\cM=(-1)^k\Omega_{\cG}(R_{\cM})$, where  
$\cG=(g_{i,j})_{i,j\in\N_0}$ and
\[
   g_{i,j}=q^{\BinomS{k+j-i}{2}}(-x)^i(-1)^j.
\]
For $q\rightarrow1$ this specializes to the well-known identity $\chi_M(x)=(-1)^kR_M(-x,-1)$ for matroids~$M$ of rank~$k$.
\end{rem}

The second instance of a monomial substitution is highly non-trivial and has been established, with very different terminology, in \cite{ByFu25}.

\begin{rem}\label{R-Tutte}
In \cite{ByFu25}, the authors introduce Tutte partitions of the subspace lattice of a \qM{}~$\cM$ (see also the precursor \cite{BCJ17}).
Such a partition, if it exists, leads to the notion of Tutte polynomial of~$\cM$.
It captures information about the intervals that make up the partition. 
The authors go on to show that the Tutte polynomial can be obtained from the Whitney function in a way that shows that 
the Tutte polynomial in fact does not depend on the choice of the Tutte partition. 
In particular, the Tutte polynomial always exists, but up to now it is unclear what information it captures in the absence of a 
Tutte partition.
We refer to \cite{ByFu25} for further details about Tutte partitions and the definition of the associated Tutte polynomial.
In \cite{ByFu25} the relation between the Tutte and Whitney polynomial is derived in form of convolution maps. 
As we will describe next, this can be expressed by way of monomial substitutions.
Let $\cM=(E,\rho)$ be a \qM{} on the ground space~$E$. 
Following \cite{ByFu25} we define 
\[
  \cA_{i,j}=\sum_{a=0}^i\sum_{b=0}^j\alpha(i,j;a,b)x^ay^b\ \text{ and }\ 
  \cB_{i,j}=\sum_{a=0}^i\sum_{b=0}^j\beta(i,j;a,b)x^ay^b\in\Z[x,y]
\]
where
\begin{align*}
    \alpha(i,j;a,b)&=\GaussianD{i}{a}_q\GaussianD{j}{b}_q q^{(i-a)(j-b)},\\
    \beta(i,j;a,b)&=(-1)^{(i-a)+(j-b)}\GaussianD{i}{a}_q\GaussianD{j}{b}_q q^{\BinomS{|(i-a)-(j-b)|}{2}}
                     \Big(1+q^{|(i-a)-(j-b)|}-q^{\max\{i-a,j-b\}}\Big),
\end{align*}
and set $\cA=(\cA_{i,j})_{i,j\in\N_0}$ and $\cB=(\cB_{i,j})_{i,j\in\N_0}$.
Then $\Omega_{\cA}\circ\Omega_{\cB}=\Omega_{\cB}\circ\Omega_{\cA}=\text{id}_{\Z[x,y]}$, that is $\Omega_{\cA}(\cB_{i,j})=x^iy^j=\Omega_{\cB}(\cA_{i,j})$, and
the convolutional relations between the Whitney function $R_\cM$ and Tutte polynomial~$T_{\cM}$ established in 
\cite[Prop.~67, Cor.~68]{ByFu25} can be rewritten as the monomial substitutions
\[
   T_{\cM}=\Omega_{\cB}(R_\cM)\ \text{ and }\ R_\cM=\Omega_{\cA}(T_\cM).
\]
This can be proven with the aid of the technical computations in \cite[Thm.~61 and its proof]{ByFu25} and some additional symmetric cases.
As illustrated in \cite[Rem.~63]{ByFu25}, for $q\rightarrow1$ (i.e., the matroid case) the above monomial substitutions specialize to the familiar identity $T_M=R_M(x-1,y-1)$ and its inverse.
\end{rem}

\section{Rank-Metric Codes and their Higher Supports and Weights}\label{S-RMCodesHigher}

In this section we will introduce rank-metric codes and various invariants. 
We first recall the analogous notions for Hamming-metric codes (i.e., block codes).

\begin{rem}\label{R-HammCodes}
The \textbf{Hamming support} of a vector $v\in\F^N$ is defined as $\suppHam(v)=\{i\in[N]\mid v_i\neq0\}$, and the \textbf{Hamming weight} 
is $\wtHam(v)=|\suppHam(v)|$. 
For any subspace $D\leq\F^N$ we set $\suppHam(D)=\bigcup_{v\in D}\suppHam(v)$ and $\wtHam(D)=|\suppHam(D)|$. 
Let $C\leq \F^N$ be an $[N,K]$-block code. For $t\in[K]_0,\,\Delta\subseteq[N]$, and $i\in[N]_0$ we set
      \begin{align*}
        A^{(t)}_{i,\text{H}}&=|\{D\leq C\mid \dim D=t,\ \wtHam(D)=i\}|,\\
        A^{(t)}_{\Delta,\text{H}}&=|\{D\leq C\mid \dim D=t,\ \suppHam(D)=\Delta\}|.
      \end{align*}
      We call $(A^{(t)}_{0,\text{H}},A^{(t)}_{1,\text{H}},,\ldots,A^{(t)}_{N,\text{H}})$ the $t$-th \textbf{higher Hamming-weight distribution} of~$C$ and
      $W^{(t)}_{C,\text{H}}=\sum_{i=0}^N A_{i,\text{H}}^{(t)}x^{N-i}y^i\in\Z[x,y]$ the $t$-th \textbf{higher Hamming-weight enumerator} of~$C$. 
      The list $\big(A^{(t)}_{\Delta,\text{H}}\big)_{\Delta\subseteq[N]}$ is called the $t$-th \textbf{higher Hamming-support distribution} of~$C$.
\end{rem}

Recall that we can associate the matroid~$M_G$ to a Hamming-metric code $C=\rs(G)$; see \cref{R-ClassMatr}.
In \cite[Thms.~3 and~4]{Bri10} it has been shown that (a) the matroid $M_G$ and the collection of all higher Hamming-support distributions of~$C$ 
determine each other, and  (b) the Whitney function $R_{M_G}$ and the collection of all higher Hamming-weight distributions of~$C$ determine each other.
In \cref{T-HigherqM} we will generalize these results to rank-metric codes and their associated \qM{}s.

\begin{defi}\label{D-RMCode}
Consider $\F_{q^m}^n$, which we may regard as an $\F$-vector space.
\begin{alphalist}
\item The \textbf{rank weight} of $v=(v_1,\ldots,v_n)\in\F_{q^m}^n$ is defined as $\wtrk(v)=\dim_{\F}\subspace{v_1,\ldots,v_n}$, where 
        $\subspace{v_1,\ldots,v_n}$ denotes the $\F$-subspace generated by $v_1,\ldots,v_n$.
\item Let $G\in\F_{q^m}^{k\times n}$. Then $\cC=\rs(G):=\{uG\mid u\in\F_{q^m}^k\}$ is called the \textbf{rank-metric code} generated by~$G$.
\item Let $\cC$ be a rank-metric code in $\F_{q^m}^n$. For $i\in[n]_0$ let $A_i=|\{v\in\cC\mid \wtrk(v)=i\}|$. 
        Then $W_\cC=\sum_{i=0}^n A_i x^{n-i}y^i\in\Z[x,y]$ is called the \textbf{rank-weight enumerator} of the code $\cC$, and 
        $\min\{i>0\mid A_i\neq0\}$ is the \textbf{minimum distance} of~$\cC$. 
\end{alphalist}
\end{defi}

For the sake of clarity we will occasionally use the following notation.

\begin{nota}\label{Nota1}
We write $V\leqq W$ (resp., $V\leqqm W$) if $V$ is an $\F$-subspace (resp., $\F_{q^m}$-subspace) of~$W$.
\end{nota}

The notions in \cref{D-RMCode} can be described with the aid of the matrix space~$\F^{m\times n}$.
Let $\psi:\F_{q^m}\longrightarrow\F^m$ be the coordinate map with respect to a fixed $\F$-basis of $\F_{q^m}$ (where we write the coordinate vectors as column vectors) and consider the $\F$-isomorphism
\begin{equation}\label{e-Psi}
   \Psi:\F_{q^m}^n\longrightarrow\F^{m\times n},\quad (v_1,\ldots,v_n)\longmapsto \begin{pmatrix}\psi(v_1)&\cdots&\psi(v_n)\end{pmatrix}.
\end{equation}
Note that $\F^{m\times n}$ is an $n$-dimensional $\F_{q^m}$-vector space with the scalar multiplication 
$aM:=\Psi(a\Psi^{-1}(M))$ for all $a\in\F_{q^m}$ and $M\in\F^{m\times n}$.
This turns~$\Psi$ into an $\F_{q^m}$-linear map.
In this notation, a rank-metric code is simply an $\F_{q^m}$-subspace of $\F^{m\times n}$.
Furthermore, for any $v\in\F_{q^m}^n$ we define the \textbf{support} of~$v$ as 
\begin{equation}\label{e-suppv}
    \supp(v)=\rs(\Psi(v))
\end{equation}
(recall the convention on $\rs(\cdot)$ from \cref{NotaIntro}). Then 
\begin{equation}\label{e-wtrkv}
    \wtrk(v)=\dim\supp(v)
\end{equation}
for all $v\in\F_{q^m}^n$. Note that $\supp(v)$ and thus $\wtrk(v)$ do not depend on the choice of basis for the coordinate map~$\psi$.
Furthermore, it is easy to see that 
\begin{equation}\label{e-PropSupp}
    \supp(\lambda v)=\supp(v)\ \text{ and }\ \supp(v+w)\leq\supp(v)+\supp(w)
\end{equation}
for all $v,w\in\F_{q^m}^n$ and $\lambda\in\F_{q^m}^*$; see also \cite[Prop.~2.3]{JuPe17}.

Recall from \cref{T-ReprqMatr} that we can associate a \qM{} $\cM_G$ to a matrix $G\in\F_{q^m}^{k\times n}$.
We can now give an alternative description of the rank function~$\rho$ of $\cM_G$.
This will be useful in the next section.

\begin{rem}\label{R-RankqPM}
Let $\cC=\rs(G)\leqqm\F_{q^m}^n$ be a rank-metric code and set $\tilde{\cC}=\Psi(\cC)\leq\F^{m\times n}$ (which is an $\F_{q^m}$-subspace of
$\F^{m\times n}$).
Let $\cM_G=(\F^n,\rho)$ be as in \cref{T-ReprqMatr} and note that~$\cM_G$ only depends on the code~$\cC$, but not on the choice of its generator matrix.
On the ground space~$\F^n$ consider the standard dot product and denote by~$V^\perp$ the associated orthogonal of the subspace $V\leqq \F^n$.
Then it is well-known (see \cite[Sec.~5]{JuPe18} or \cite[Lem.~4.2]{GLJ22Gen}) that for any $V\leqq\F^n$
\begin{equation}\label{e-rhoCtilde}
    \rho(V)=\frac{\dim_{\F}\tilde{\cC}-\dim_{\F}\tilde{\cC}(V^\perp)}{m},\ \text{ where }
    \tilde{\cC}(V^\perp)=\{M\in\tilde{\cC}\mid \rs(M)\leqq V^\perp\}.
\end{equation}
Since $\tilde{\cC}$ and $\tilde{\cC}(V^\perp)$ are $\F_{q^m}$-subspaces, the dimensions in the numerator above are multiples of~$m$ and 
the fraction is indeed an integer.
\end{rem}

We now define the rank-metric analogues of the  distributions introduced in \cref{R-HammCodes}. To do so, we generalize \eqref{e-suppv} and \eqref{e-wtrkv} to subspaces and define for any $T\leqqm\F_{q^m}^n$
\begin{equation}\label{e-SupportT}
   \supp(T)=\sum_{v\in T}\rs(\Psi(v)) \ \text{ and }\ \wtrk(T)=\dim\supp(T).
\end{equation}

\begin{defi}\label{D-Support}
Let $\cC\leq_{q^m}\F_{q^m}^n$ be a rank-metric code and $\dim_{\F_{q^m}}\cC=k$. For $t\in[k]_0$ define
\[
    \cD_t(\cC)=\{\cC'\leq_{q^m}\cC\mid \dim_{\F_{q^m}}\cC'=t\}.
\]
Furthermore, for $i\in[n]_0$ and for $V\leqq\F^n$ define 
\[
   A_i^{(t)}=\big|\{\cC'\in\cD_t(\cC)\mid \wtrk(\cC')=i\}\big|\ \; \text{ and }\ \;
   A_V^{(t)}=\big|\{\cC'\in\cD_t(\cC)\mid \supp(\cC')=V\}\big|.
\]
We call $(A_0^{(t)},\ldots,A_n^{(t)})$ the \textbf{$t$-th higher rank-weight distribution},
and the polynomial $W_{\cC}^{(t)}=\sum_{i=0}^nA_i^{(t)}x^{n-i}y^i\in\Z[x,y]$ is the  \textbf{$t$-th higher rank-weight enumerator} of~$\cC$.
Furthermore, $\big(A_V^{(t)}\big)_{V\leq\F^n}$ is the \textbf{$t$-th higher rank-support distribution} of~$\cC$.
Except for \cref{S-Proj}, we will skip the word `rank'  and simply use `higher weight distribution' and `higher support distribution'.
\end{defi}

Note that $(A^{(0)}_0,\ldots,A^{(0)}_n)=(1,0,\ldots,0)$ and $A_i^{(1)}=A_i/(q^m-1)$ for all $i>0$, where~$A_i$ are the weight enumerators 
introduced in \cref{D-RMCode}(c).

The distributions defined in \cref{R-HammCodes} and \cref{D-Support} appear in the literature under various names.
Consider first rank-metric codes.
In \cite[Def.~2.4]{JuPe17} the invariants $A_i^{(t)}$ are denoted by $A_W^{R,t}$, and $W_{\cC}^{(t)}$  is called 
the $t$-th generalized rank weight enumerator.
On the other hand, in \cite[Def.~2.5]{JuPe17} and \cite[Def.~5.4]{Gor21} the $t$-th generalized weight of~$\cC$ is defined as the minimum~$i>0$ for 
which $A_i^{(t)}\neq0$. 
For Hamming-metric codes, the higher Hamming-support distribution is called support weight distribution in \cite[p.~166]{Ba97} and subcode support multiplicities in \cite[p.~4351 and Thm.~3]{Bri10}.
The higher Hamming-weight enumerator corresponds to the higher weight enumerator in \cite[p.~4350]{Bri10} and the generalized weight 
enumerator in \cite[Sec.~5.5.1]{JuPe13}.

\section{$q$-Matroids and Higher Distributions}\label{S-WhitHigh}

In this section we will restrict ourselves to representable \qM{}s. 
We will see that the \qM{} associated to a rank-metric code carries the same information as the collection of all higher support distributions 
of the code, while the Whitney function carries the same information as the collection of all higher weight enumerators.
In fact, these invariants are related via monomial substitutions. 
Because of the $\Z$-linearity of the relations, it appears most convenient to use a matrix-theoretical approach, as in \cite{Bri10} for Hamming-metric codes.

Recall the M{\"o}bius function on the subspace lattice $\cL(E)$ given by 
\begin{equation}\label{e-muWV}
    \mu(W,V)=\left\{\begin{array}{cl} (-1)^{v-w}q^{\BinomS{v-w}{2}},&\text{if }W\leq V,\\[.5ex] 0,&\text{otherwise,}\end{array}\right\},\quad
    \text{ where }v=\dim V,\ w=\dim W.
\end{equation}
We also remind of the Kronecker-$\delta$ function from \cref{NotaIntro}.

\begin{defi}\label{D-Matrices}
Let $\cM=(E,\rho)$ be a \qM{} of rank~$k$ on an $n$-dimensional ground space~$E$.
Let $\nu_{i,j}$ be as in \eqref{e-WhitExpand}, \eqref{e-WhitCoeff} and set $\nu_{i,b}:=0\text{ for }b\not\in\{0,\ldots,n-k\}$.
Set $\hat{n}=|\cL(E)|=\sum_{i=0}^n\Gaussian{n}{i}_q$ and fix some $m\in\N$. 
Moreover, choose a non-degenerate symmetric bilinear form (NSBF) on~$E$.
We define the matrices
\begin{align*}
     &&P&=\big(\delta_{W,V^\perp}\big)_{W,V\in\cL(E)}\in\Z^{\hat{n}\times\hat{n}}, &&
          M=\big(\mu(W,V)\big)_{V,W\in\cL(E)}\in\Z^{\hat{n}\times\hat{n}},\\[.7ex]
     &&F&=\Big((-1)^{i-j}q^{\BinomS{i-j}{2}}\Gaussian{n-j}{i-j}_q\Big)_{i,j\in[n]_0}\in\Z^{(n+1)\times(n+1)}, &&
          \,Q=\Big(\Gaussian{t}{j}_{q^m}\Big)_{t,j\in[k]_0}\in\Z^{(k+1)\times(k+1)},\\[.7ex]
     &&N&=\big(\nu_{i,n-k+i-j}\big)_{i\in[k]_0\atop j\in[n]_0}\in\Z^{(k+1)\times(n+1)}, &&
          \,D=\big(\delta_{i,\dim V}\big)_{i\in[n]_0\ \atop V\in\cL(E)}\in\Z^{(n+1)\times\hat{n}},\\[.7ex]
     &&K&=\big(\delta_{k-\rho(V),\,j})_{V\in\cL(E)\atop j\in[k]_0\ }\in\Z^{\hat{n}\times(k+1)}.&&
\end{align*}
\end{defi}

Some comments are in order.
First, only the matrix~$P$ depends on the chosen NSBF.
Second, the matrix~$K$ contains the same information as the rank function~$\rho$ because $K_{V,j}=1$ iff $\rho(V)=k-j$.
Finally, the matrix~$N$ carries exactly the same information as the Whitney function. 
This will also be made precise in \cref{P-MatRelations}(a) below.

\begin{prop}\label{P-Invertible}
The matrices $P,M,F,$ and~$Q$ are invertible over~$\Z$. 
Precisely, $P^{-1}=P$ and 
\[
   M^{-1}=\big(\hat{M}_{V,W}\big)_{V,W},\ \;  Q^{-1}=\Big((-1)^{j-t}q^{m\BinomS{j-t}{2}}\Gaussian{j}{t}_{q^m}\Big)_{j,t},\ \;
   F^{-1}=\Big(\Gaussian{n-j}{i-j}_q\Big)_{i,j},
\]
where $\hat{M}_{V,W}=1$ if $W\leq V$ and $\hat{M}_{V,W}=0$ otherwise.
\end{prop}

\begin{proof}
First of all,  $(PP)_{V,W}=\sum_Z \delta_{V,Z^\perp} \delta_{Z,W^\perp}=\delta_{V,W}$ and thus $P=P^{-1}$.
Next, note that $\hat{M}_{V,W}$ is the $\zeta$-function on $\cL(E)$, which is the inverse of the M\"obius function, and we compute
$(\hat{M}M)_{V,W}=\sum_{W\leq Z\leq V}\hat{M}_{V,Z}M_{Z,W}=\delta_{V,W}$.
The statement about $Q^{-1}$ follows from the orthogonality relations for $q$-binomial coefficients (applied to $q^m$). Indeed,
\[
   (QQ^{-1})_{t,i}=
   \sum_{j=0}^{k}(-1)^{j-i}q^{m\BinomS{j-i}{2}}\GaussianD{t}{j}_{q^m}\GaussianD{j}{i}_{q^m}
   =\sum_{j=i}^{t}(-1)^{j-i}q^{m\BinomS{j-i}{2}}\GaussianD{t}{j}_{q^m}\GaussianD{j}{i}_{q^m}=\delta_{t,i};
\]
see also \cite[Identity~(12)]{Bri10}.
Finally, with
\[
    U=\begin{pmatrix} & &1\\ &\iddots& \\ 1& &\end{pmatrix}\in\Z^{(n+1)\times(n+1)}
\]
we obtain $(UFU)_{a,b}=(-1)^{b-a}q^{\BinomS{b-a}{2}}\Gaussian{b}{a}$ and thus $(UFU)^{-1}=\big(\Gaussian{b}{a}\big)_{a,b}$ by the orthogonality relations.
Thus $F^{-1}=U(\Gaussian{b}{a})_{a,b}U=\big(\Gaussian{n-b}{n-a}\big)_{a,b}$, as desired.
\end{proof}

\begin{prop}\label{P-MatRelations}\
\begin{alphalist}
\item $uN\T v\T=R_{\cM}(x,y)$, where  $u=(y^n,y^{n-1},\ldots,1)$ and $v=(y^{-k},xy^{-k+1},x^2y^{-k+2},\ldots,x^{k})$.
\item $DPK=N\T$.
\item $FD=DM$.
\end{alphalist}
\end{prop}

\begin{proof}
(a) can easily be verified.
(b) $(DP)_{i,V}=\sum_{W\leq E}\delta_{i,\dim W}\delta_{W,V^\perp}$, which is~$1$ iff $\dim V=n-i$ and~$0$ otherwise.
Thus we obtain
\begin{align*}
   (DPK)_{i,j}&=\sum_{V\leq E}(DP)_{i,V}K_{V,j}=\sum_{\dim V=n-i}\delta_{k-\rho(V),j}\\
   &=\big|\{V\leq E\mid \dim V=n-i,\,\rho(V)=k-j\}\big|=\nu_{j,n-k+j-i}=N_{j,i}=(N\T)_{i,j}.
\end{align*}
(c) Let $V\leq E$ and $\dim V=v$. We compute
$(FD)_{i,V}=\sum_{j=0}^n(-1)^{i-j}q^{\BinomS{i-j}{2}}\Gaussian{n-j}{i-j}_q\delta_{j, v}=(-1)^{i-v}q^{\BinomS{i-v}{2}}\Gaussian{n-v}{i-v}_q$.
On the other hand,
$(DM)_{i,V}=\sum_{W}\delta_{i,\dim W}\mu(V,W)=\sum_{\dim W=i}\mu(V,W)=\sum_{\dim W=i\atop V\leq W}(-1)^{i-v}q^{\BinomS{i-v}{2}}=
   (-1)^{i-v}q^{\BinomS{i-v}{2}}\Gaussian{n-v}{n-i}_q$,
which proves the desired identity.
\end{proof}

Now we turn to representable \qM{}s. 
The following theorem shows for a rank-metric code $\cC=\rs(G)$ and its associated \qM{} $\cM_G$ that 
(a) the collection of all higher support distributions of $\cC$ and the \qM{} $\cM_G$ determine each other, and 
(b) the collection of all higher weight distributions of~$\cC$ and the Whitney function of $\cM_G$ determine each other.
These are the $q$-analogues of \cite[Thms.~3 and~4]{Bri10}. 

\begin{theo}\label{T-HigherqM}
Let $G\in\F_{q^m}^{k\times n}$ be of rank~$k$ and $\cC=\rs(G)$. 
Let $\cM_G=(\F^n,\rho)$ be the associated \qM{} as in \cref{T-ReprqMatr}.
Define
\[
   A=\big(A_i^{(t)}\big)_{i\in[n]_0,t\in[k]_0}\in\Z^{(n+1)\times(k+1)}\ \text{ and }\  S=\big(A_V^{(t)}\big)_{V\in\cL(\F^n),t\in[k]_0}\in\Z^{\hat{n}\times(k+1)},
\]
where $A_i^{(t)}$ and $A_V^{(t)}$ are as in \cref{D-Support}.
As NSBF on~$\F^n$ choose the standard dot product and consider the matrices in \cref{D-Matrices}. Then
\begin{alphalist}
\item  $S=MPKQ$ and thus $K=PM^{-1}SQ^{-1}$.
           In other words, the \qM{} $\cM_G$ (i.e., the matrix~$K$) and the collection of all higher support distributions (i.e., the matrix~$S$) 
           determine each other.
\item $A=FN\T Q$ and thus $N\T=F^{-1}AQ^{-1}$. 
           In other words, the Whitney function (i.e., the matrix~$N$) and the collection of all higher weight distributions (i.e., the matrix~$A$) determine each other.
\end{alphalist}
\end{theo}

For the proof we first establish the following lemma.

\begin{lemma}\label{L-qMSupp}
Let the data be as in \cref{T-HigherqM}. For all $V\leq\F^n$ and $t\in[k]_0$ we have
\[
     A_V^{(t)}=\sum_{W\leq V}(-1)^{v-w}q^{\BinomS{v-w}{2}}\GaussianD{k-\rho(W^\perp)}{t}_{q^m}
     =\sum_{j=t}^k\sum_{W\leq V}(-1)^{v-w}q^{\BinomS{v-w}{2}}\GaussianD{j}{t}_{q^m}\delta_{k-\rho(W^\perp),j}
\]
where $v=\dim V$ and $w=\dim W$.
\end{lemma}

\begin{proof}
Fix $V\leq\F^n$ and $t\in[k]_0$.
In order to prove the first identity, we start with
\[
   \sum_{W\leq V}A_W^{(t)}=\sum_{W\leq V}\big|\{\cC'\in\cD_t(\cC)\mid \supp(\cC')=W\}\big|
                =\big|\{\cC'\in\cD_t(\cC)\mid \supp(\cC')\leq V\}\big|.
\]
The set $\{\cC'\in\cD_t(\cC)\mid \supp(\cC')\leq V\}$ equals $\cD_t(\Psi^{-1}(\tilde{\cC}(V)))$, where $\tilde{\cC}$ and $\tilde{\cC}(V)$ are 
as in  \cref{R-RankqPM} (here we need the standard dot product on~$\F^n$ as NSBF).
Thus its cardinality is the number of $t$-dimensional $\F_{q^m}$-subspaces of 
$\tilde{\cC}(V)$.
Setting $d=\dim_{\F_q}\tilde{\cC}(V)$ and using \eqref{e-rhoCtilde}, we arrive at
\[
    \sum_{W\leq V}A_W^{(t)}=\GaussianD{\dim_{\F_{q^m}}\tilde{\cC}(V)}{t}_{q^m}=\GaussianD{d/m}{t}_{q^m}
    =\GaussianD{k-\rho(V^\perp)}{t}_{q^m}=:f(V).
\]
Using M\"obius inversion we obtain
\[
  A_V^{(t)}=\sum_{W\leq V}\mu(W,V)f(W)= \sum_{W\leq V}(-1)^{v-w}q^{\BinomS{v-w}{2}}\GaussianD{k-\rho(W^\perp)}{t}_{q^m},
\]
where $w=\dim W$. 
This establishes the first identity. The second one is clear.
\end{proof}

\textit{Proof of \cref{T-HigherqM}:}
\\
(a) Using lower case letters for the dimension of spaces with corresponding upper case letters, we have
\begin{align*}
  (MP)_{V,Z^\perp}&=\sum_W \mu(W,V)\delta_{W,Z}=\sum_{W\leq V}(-1)^{v-w}q^{\BinomS{v-w}{2}}\delta_{W,Z}
          =\left\{\begin{array}{cl} (-1)^{v-z}q^{\BinomS{v-z}{2}},&\text{if }Z\leq V,\\ 0,&\text{otherwise}\end{array}\right.\\[1ex]
   (KQ)_{Z^\perp,t}&=\sum_{j=0}^k\delta_{k-\rho(Z^\perp),j}\GaussianD{j}{t}_{q^m}.
\end{align*}
Thus,
$(MPKQ)_{V,t}=\sum_{Z\leq V}\sum_{j=0}^k(-1)^{v-z}q^{\BinomS{v-z}{2}}\Gaussian{j}{t}_{q^m}\delta_{k-\rho(Z^\perp),j}=A_V^{(t)}$,
where the last step follows from \cref{L-qMSupp}.
The rest follows from \cref{P-Invertible}.
\\
(b) With~$D$ as in \cref{D-Matrices}, we have
$(DS)_{i,t}=\sum_{V\leq\F^n}\delta_{i,\dim V}A_V^{(t)}=\sum_{\dim V=i}A_V^{(t)}=A_i^{(t)}$. 
Thus $DS=A$, and  with \cref{P-MatRelations}  and Part~(a) we compute $A=DS=DMPKQ=FDPKQ=FN\T Q$.
Again, the rest follows from \cref{P-Invertible}.
\hfill $\square$

We can rephrase the identity in \cref{T-HigherqM}(b) as monomial substitutions.

\begin{cor}\label{C-SubA}
Let the data be as in \cref{T-HigherqM}.
For $t\in[k]_0$ consider the $t$-th higher weight enumerator $W^{(t)}_{\cC}\in\Z[x,y]$.
\begin{alphalist}
\item $W^{(t)}_{\cC}=\Omega_{\cU^{(t)}}(R_\cM)$, where $\cU^{(t)}=(u_{a,c}^{(t)})_{a,c\in\N_0}$ and 
      \[
         u_{a,c}^{(t)}=\sum_{i=0}^n (-1)^{i-n+k-a+c}q^{\BinomS{i-n+k-a+c}{2}}\GaussianD{k-a+c}{n-i}_q\GaussianD{a}{t}_{q^m} x^{n-i}y^i
         \ \ \text{ if }a\in[k]_0,\,c\in[n-k]_0
      \]
      and $u_{a,c}^{(t)}=0$ otherwise.
\item $R_\cM=\sum_{t=0}^k\Omega_{\widehat{U}^{(t)}}(W^{(t)}_{\cC})$, where 
      $\widehat{\cU}^{(t)}=(\hat{u}^{(t)}_{i,j})_{i,j\in\N_0}$ and 
      \[
        \hat{u}^{(t)}_{i,j}=\sum_{a=0}^k\sum_{c=0}^{n-k}(-1)^{t-a}q^{m\BinomS{t-a}{2}}\GaussianD{t}{a}_{q^m}\GaussianD{n-j}{k+c-a}_qx^ay^c
        \ \ \text{ if }n-i=j\in[n]_0
      \]
      and $\hat{u}^{(t)}_{i,j}=0$ otherwise. 
\end{alphalist}
\end{cor}

\begin{proof}
We have $W_{\cC}^{(t)}=\sum_{i=0}^nA_i^{(t)}x^{n-i}y^i$.
\\
(a) The identity $A=FN\T Q$ from \cref{T-HigherqM}(b) leads to 
\[
    A_i^{(t)}=\sum_{a=0}^k\sum_{b=0}^n(-1)^{i-b}q^{\BinomS{i-b}{2}}\GaussianD{n-b}{i-b}_q\GaussianD{a}{t}_{q^m}\nu_{a,n-k+a-b}
\]
for all $i\in[n]_0$ and $t\in[k]_0$.
(Note that the inner term is nonzero only if $b\in\{a,\ldots,n-k+a\}$.)
With the index change $c=n-k+a-b$ we obtain 
\begin{equation}\label{e-Ait}
    A_i^{(t)}=\sum_{a=0}^k\sum_{c=0}^{n-k}(-1)^{i-n+k-a+c}q^{\BinomS{i-n+k-a+c}{2}}\GaussianD{k-a+c}{n-i}_q\GaussianD{a}{t}_{q^m}\nu_{a,c}
\end{equation}
for all $i,t$, and the result follows. 
\\
(b) Thanks to \cref{T-HigherqM}(b) we have $N=(Q^{-1})\T A\T(F^{-1})\T$. Using that $\nu_{a,c}=N_{a,n-k+a-c}$ and \cref{P-Invertible}
we obtain 
\begin{equation}\label{e-nuac}
   \nu_{a,c}=\sum_{t=0}^k\sum_{j=0}^n(-1)^{t-a}q^{m\BinomS{t-a}{2}}\GaussianD{t}{a}_{q^m}\GaussianD{n-j}{k+c-a}_qA_j^{(t)}
      \ \ \text{ for all }a=[k]_0,\,c\in[n-k]_0.
\end{equation}
Thus
\begin{align*}
   R_\cM&=\sum_{a=0}^k\sum_{c=0}^{n-k} \nu_{a,c}x^ay^c
      =\sum_{t=0}^k\sum_{j=0}^n A^{(t)}_j
           \sum_{a=0}^k\sum_{c=0}^{n-k}(-1)^{t-a}q^{m\BinomS{t-a}{2}}\GaussianD{t}{a}_{q^m}\GaussianD{n-j}{k+c-a}_qx^ay^c\\
       &=\sum_{t=0}^k\sum_{j=0}^nA_j^{(t)}\hat{u}_{n-j,j}=\sum_{t=0}^k\Omega_{\widehat{\cU}^{(t)}}(\W_\cC^{(t)}).\qedhere
\end{align*}
\end{proof}

We briefly compare  the relations in \cref{C-SubA} with an earlier result in the literature which extracts the weight enumerator from 
a 4-variable Whitney function.

\begin{rem}\label{R-Shi19}
In \cite[p.~1768]{Shi19} the author studies $q$-polymatroids and introduces a 4-variable Whitney function
$R(X_1,X_2,X_3,X_4)$. For \qM{}s,  $R(X,Y,1,0)$ equals the Whitney function as in \cref{D-Polys}.
It is then shown \cite[Thm.~14]{Shi19} that for a representable $q$-polymatroid a suitable variable substitution of $R(X_1,X_2,X_3,X_4)$ 
results in the rank-weight enumerator $W_{\cC}$ of the representing code; see \cref{D-RMCode}(c). 
\cref{C-SubA} shows that with suitable monomial substitutions the `standard' two-variable Whitney function leads to 
all $t$-th higher weight enumerators and vice versa.
\end{rem}

\cref{C-SubA} may be regarded as a $q$-analogue of \cite[Thm.~9]{Bri07} and \cite[Thm.~5.11]{JuPe13}.
Let us provide some details. 

\begin{rem}\label{R-JuPe13Bri}
If in \cref{C-SubA}(a)  we replace $\Gaussian{k-a+c}{n-i}_q$ by $\Binom{k-a+c}{n-i}$ and~$q^m$ by~$q$ and omit $q^{\BinomS{i-n+k-a+c}{2}}$,
we obtain the Hamming-metric case as derived in \cite[Eq.~(20)]{Bri10}. 
It can be written as a classical substitution; see \cite[Thm.~2]{Bri10},  \cite[Thm.~9]{Bri07} or \cite[Thm.~5.11]{JuPe13}.
This result generalizes the celebrated Greene's Theorem \cite[Cor.~4.5]{Gre76}, which expresses the `standard' weight enumerator ($t=1$) in terms of the Whitney function.
Likewise, if in~(b) we replace~$q^m$ by~$q$ and $\Gaussian{n-j}{k+c-a}_q$ by $\Binom{n-j}{k+c-a}$, we obtain the Hamming-metric case as 
derived in \cite[Thm.~5; see also Rem.~6]{Bri10} and  \cite[Thm.~5.11]{JuPe13}.
The results in \cite{Bri07,JuPe13} have been achieved with the aid of an extended weight enumerator, involving an additional variable.
\end{rem}

\begin{exa}\label{E-53}
Let $\F=\F_2$ and consider the field $\F_{2^7}$ with primitive element~$\omega$ satisfying $\omega^7+\omega+1=0$. 
For $i=1,2$ define the \qM{}s $\cM_i=\cM_{G_i}=(\F^5,\rho_i)$, where  
\[
   G_1=\begin{pmatrix}1&0&0&\omega^{65}&\omega^{85}\\0&1&0&\omega^{37}&\omega^{72}\\
                       0&0&1&\omega^{124}&\omega^{118}\end{pmatrix},\
   G_2=\begin{pmatrix}1&0&0&\omega^{26}&\omega^{64}\\0&1&0&\omega^{27}&\omega^{20}\\
                       0&0&1&\omega^{50}&\omega^{92}\end{pmatrix}\in\F_{2^7}^{3\times5}.
\]                       
In \cite[Ex.~5.9]{GLJ25CF} it has been shown that 
\begin{equation}\label{e-WhitExa}
  R_{\cM_1}=R_{\cM_2}=x^3 + 31x^2 + (3y + 155)x + y^2 + 31y + 152.
\end{equation}
Thus, \cref{T-HigherqM}(b) implies that the rank-metric codes $\cC_i=\rs(G_i)$  share all higher weight enumerators. 
Indeed, the matrix~$A$ defined in that theorem is (for both codes) given by
\[
   A=\begin{pmatrix}1 & 0 & 0 & 0 \\0 & 0 & 0 & 0 \\0 & 3 & 0 & 0 \\0 & 134 & 0 & 0 \\0 & 3576 & 31 & 0 \\0 & 12800 & 16482 & 1\end{pmatrix}.
\]
For instance, $W_{\cC_i}^{(1)}=3x^{3}y^2+134x^2y^3+3576xy^4+12800y^5$ and thus the weight enumerator, defined  in \cref{D-RMCode}(c),
 is
$W_{\cC_i}=x^5+(2^7-1)W_{\cC_i}^{(1)}=x^5+381x^{3}y^2+17018x^2y^3+454152xy^4+1625600y^5$.
In \cite[Ex.~5.9]{GLJ25CF} it has been shown that the two \qM{}s are not isomorphic (see \cref{D-Equiv}).
As a consequence, there exists no $U\in\GL_5(\F)$ such that $\cC_1=\cC_2'$, where $\cC_2'=\rs(G_2U)$.
Even more, \cref{T-HigherqM}(a) tells us there is no $U\in\GL_5(\F)$ such that $\cC_1$ and $\cC_2'$ share all higher support distributions.
This can also be verified as follows.
Since $A_2^{(1)}=3$, both~$\cC_1$ and~$\cC_2$ have three 1-dimensional subcodes with rank weight~$2$.
For~$\cC_1$ these subcodes are generated by
\[
    \begin{pmatrix}1&1&0&w^{41}&1\end{pmatrix},\
    \begin{pmatrix}1&w^{77}&1&1&w^{77}\end{pmatrix},\
     \begin{pmatrix}1&w^{103}&w^{103}&w^{103}+1&0\end{pmatrix},
 \]
and their supports (see \eqref{e-SupportT})  are given by 
\[
  V_1=\rs\begin{pmatrix}1 & 1 & 0 & 0 & 1 \\0 & 0 & 0 & 1 & 0\end{pmatrix},\;
  V_2=\rs\begin{pmatrix}1 & 0 & 1 & 1 & 0 \\0 & 1 & 0 & 0 & 1\end{pmatrix},\;
  V_3=\rs\begin{pmatrix}1 & 0 & 0 & 1 & 0 \\0 & 1 & 1 & 1 & 0\end{pmatrix}.
\]
For~$\cC_2$ the subcodes are generated by 
\[
   \begin{pmatrix}1&w^{44}&0&w^{44}+1&0\end{pmatrix},\
    \begin{pmatrix}1&w^{83}&1&w^{83}&w^{83}+1\end{pmatrix},\
    \begin{pmatrix}1&1&w^{77}&w^{77}&1\end{pmatrix},
\]
and their supports are
\[
  W_1=\rs\begin{pmatrix}1 & 0 & 0 & 1 & 0 \\0 & 1 & 0 & 1 & 0\end{pmatrix},\;
  W_2=\rs\begin{pmatrix}1 & 0 & 1 & 0 & 1 \\0 & 1 & 0 & 1 & 1\end{pmatrix},\;
  W_3=\rs\begin{pmatrix}1 & 1 & 0 & 0 & 1 \\0 & 0 & 1 & 1 & 0\end{pmatrix}.
\]
Noticing that $\dim(V_1+V_2+V_3)=5$, whereas $\dim(W_1+W_2+W_3)=4$, we conclude that there is no isomorphism $U\in\GL_5(\F)$
such that $\cC_1$ and $\cC_2'=\rs(G_2U)$ share all higher support distributions, and this confirms that the \qM{}s are not isomorphic.
\end{exa}

\begin{rem}\label{R-Ex53}
Consider again \cref{E-53} and denote by $(A_V^{(t)})_{V\in\cL(\F^5)}$ and $(B_V^{(t)})_{V\in\cL(\F^5)}$ the $t$-th higher support 
distributions of~$\cM_1$ and~$\cM_2$, respectively. Let
\[
   W=\subspace{e_1+e_4,\,e_2+e_4,\,e_3+e_4,\,e_5},
\]
which is the subspace $W_1+W_2+W_3$ from \cref{E-53}.
We have $A_W^{(1)}=114$ whereas $B_W^{(1)}=120$. Even more, there is no subspace~$V$ such that $A_V^{(1)}=120$.
This confirms that~$\cM_1$ and~$\cM_2$ are not isomorphic. 
Moreover, this implies that there are $114(q^m-1)=14478$ vectors in~$\cC_1$ and $120(q^m-1)=15240$ vectors in~$\cC_2$ 
with support~$W$.
This can also be determined with the aid of the Critical Theorem, see \cite[Thm.~6.20]{Ja23} or \cite[Thm.~4.3]{AlBy23}. Indeed,
the characteristic polynomials of the contractions $\cM_i/W^\perp$ evaluate to $\chi_{\cM_1/W^\perp}(q^m)=14478$ and
$\chi_{\cM_2/W^\perp}(q^m)=15240$.
\end{rem}

We conclude this section with some consequences of \cref{C-SubA} and its proof.
First, the non-negativity of the coefficients of the Whitney function and the higher weight enumerators lead to the following.

\begin{rem}\label{C-NonNega}
With the data as in \cref{T-HigherqM}, the identities \eqref{e-Ait} and \eqref{e-nuac} imply
\begin{align}
    &\sum_{a=0}^k\sum_{c=0}^{n-k}(-1)^{i-n+k-a+c}q^{\BinomS{i-n+k-a+c}{2}}\GaussianD{k-a+c}{n-i}_q\GaussianD{a}{t}_{q^m}\nu_{a,c}
    \geq 0\ \ \text{ for all }(i,t)\in[n]_0\times[k]_0, \label{e-Ineq1}\\
    &\sum_{t=0}^k\sum_{j=0}^n(-1)^{t-a}q^{m\BinomS{t-a}{2}}\GaussianD{t}{a}_{q^m}\GaussianD{n-j}{k+c-a}_qA_j^{(t)}\geq 0
    \ \ \text{ for all }(a,c)\in[k]_0\times[n-k]_0 \nonumber.
\end{align}
The inequalities in \eqref{e-Ineq1} provide us with conditions for the coefficients of the Whitney function of a \qM{}.
Since they do not involve the rank-metric code itself, but only the field size~$q^m$, 
they give us a necessary condition on the size of the smallest field over which there exists a rank-metric code such that its associated 
\qM{} has the given Whitney function.
This condition is in general not sufficient.
For instance, for the Whitney function in \eqref{e-WhitExa} the smallest~$m$ such that \eqref{e-Ineq1} is true is 
$m=4$ (where $q=2$), but one easily verifies that the smallest field size for which there exists a rank-metric code such that the 
Whitney function is as in \eqref{e-WhitExa} is  $\F_{2^5}$.
Even more, the \qM{} $\cM_1$ can be represented over $\F_{2^5}$, whereas~$\cM_2$ needs~$\F_{2^6}$.
\end{rem}

Next, the case $t=0$ in \cref{C-SubA}(b) gives rise to some interesting identities for the coefficients $\nu_{a,c}$ of the Whitney function.
They are also true for non-representable \qM{}s.

\begin{cor}\label{C-WhitCoeffRelation}
Let $\cM$ be a \qM{} of rank~$k$ on an $n$-dimensional ground space and $R_\cM$ be as in \eqref{e-WhitExpand}.
Then 
\[
   \sum_{a=0}^k\sum_{c=0}^{n-k}\nu_{a,c}(-1)^{k+c-a-\ell}q^{\BinomS{k+c-a-\ell}{2}}\Gaussian{k+c-a}{\ell}_q=\delta_{\ell,n}
   \ \text{  for all }\ \ell\in[n]_0.
\]
\end{cor}

\begin{proof}
For representable \qM{}s this follows from \eqref{e-Ait} because $A_i^{(0)}=\delta_{i,0}$.
For general \qM{}s we need to go back to the proof of \cref{T-HigherqM}.
Define $\hat{S}=MPKQ$. 
Thus, if~$\cM$ is representable, then $\hat{S}=S$, defined in \cref{T-HigherqM}. 
Writing $\hat{S}=(\hat{S}_{V,t})_{V,t}$ we obtain  as in the proof of \cref{T-HigherqM}(a)
$\hat{S}_{V,0}=\sum_{Z\leq V}\sum_{j=0}^k(-1)^{v-z}q^{\BinomS{v-z}{2}}\delta_{k-\rho(V^\perp),j}$.
Hence 
\[
   \hat{S}_{V,0}=\sum_{Z\leq V}(-1)^{v-z}q^{\BinomS{v-z}{2}}\sum_{j=0}^k\delta_{k-\rho(V^\perp),j}
   =\sum_{Z\leq V}\mu(Z,V)=\delta_{V,0}\ \text{ for all }\ V\in\cL(E).
\]
Next set $\hat{A}=D\hat{S}$. 
Then $\hat{A}_{i,0}=\sum_{V\leq E}\delta_{i,\dim V}\hat{S}_{V,0}=\delta_{i,0}$ for all $i\in[n]_0$.
Furthermore, as in the proof of \cref{T-HigherqM}(b) we have $\hat{A}=FN\T Q$ and thus 
\eqref{e-Ait} for $t=0$ implies
$\delta_{i,0}=\sum_{a=0}^k\sum_{c=0}^{n-k}(-1)^{i-n+k-a+c}q^{\BinomS{i-n+k-a+c}{2}}\Gaussian{k-a+c}{n-i}_q\nu_{a,c}$.
The result follows with the index change $\ell=n-i$.
\end{proof}

\section{The Projectivization Matroid}\label{S-Proj}
In this section we will consider the projectivization matroid of a \qM{}. 
These matroids have been introduced in \cite{JPV23} and studied in detail in \cite{Ja23}. 
We will show that the Whitney function of a \qM{} and that of its projectivization determine each other via monomial substitutions.

Recall from \cref{NotaIntro} that $\subspace{v_1,\ldots,v_t}$ denotes the $\F$-vector space generated by $v_1,\ldots,v_t$.

\begin{nota}\label{Nota}
For $n\in\N_0$ set $\subspace{n}_q:=\frac{q^n-1}{q-1}$ and $\cR_n=\big[\subspace{n}_q\big]$, thus 
$\cR_n=\{1,\ldots,\subspace{n}_q\}$.
For an $n$-dimensional vector space~$E$ over~$\F$ denote by $\P E$ the projective space over~$E$.
Thus, $\P E=\{\subspace{v_1},\ldots,\subspace{v_{\subspace{n}_q}}\}$ for any generators~$v_i$ of the distinct lines in $E$ 
(if $\dim E=0$, then $\P E=\emptyset$).
In particular, $\subspace{n}_q=|\P E|$.
Furthermore, for a subset $A=\{\subspace{v_{i_1}},\ldots,\subspace{v_{i_t}}\}\subseteq\P E$ define $\subspace{A}:=\subspace{v_{i_1},\ldots,v_{i_t}}$, that is $\subspace{A}$ is the subspace in~$E$ generated by the lines in~$A$.
Throughout this section, it will be convenient to consider the vectors $v_j\in\F^n$ as column vectors.
\end{nota}

From now on, fix an integer $n\in\N$ and an $n$-dimensional $\F$-vector space~$E$.
The following projectivization matroids have been introduced in \cite[Def.~14 and Prop.~15]{JPV23}
and studied in detail in \cite{Ja23}.

\begin{defi}\label{D-ProjectqM}
Let $\cM=(E,\rho)$ be a \qM. 
The \textbf{projectivization} of~$\cM$, denoted by $\P\cM$,  is defined as the matroid with ground space $\P E$ and 
rank function given by $r(A)=\rho(\subspace{A})$ for all $A\subseteq\P E$.
Note that~$\cM$ and~$\P\cM$ have the same rank.
\end{defi}

Let $q$-$\Mat_{k,n}$ be the set of all \qM{}s of rank~$k$ on an $n$-dimensional ground space
and $\Mat_{k,n}$ be the set of all matroids of rank~$k$ on a ground set of cardinality~$n$. 
Then projectivization is an injective map from $q$-$\Mat_{k,n}$  to $\Mat_{k,\subspace{n}_q}$.
In \cite[Ch.~4]{Ja23} it is shown that this map is even a functor between suitable categories of \qM{}s and matroids.

Indexing the lines in~$\P E$, we may write the projectivization as follows.

\begin{rem}\label{R-ProjRn}
Let~$\cM$  be as in \cref{D-ProjectqM}.
Choose $v_1,\ldots,v_{\subspace{n}_q}\in E$ such that $\P E=\{\subspace{v_1},\ldots,\subspace{v_{\subspace{n}_q}}\}$.
Then $\hat{M}=(\cR_n,\hat{r})$ with $\hat{r}(\{i_1,\ldots,i_t\})=\rho(\subspace{v_{i_1},\ldots,v_{i_t}})$ 
is a matroid, and it satisfies $\hat{M}\cong \P\cM$.
\end{rem}

In order to describe the projectivization of representable \qM{}s we cast the following terminology.
The  Hamming-metric code~$\hat{\cC}$ below has also been studied in \cite[Sec.~4]{ABNR22}.

\begin{defi}\label{D-ProjCode}
Choose $v_1,\ldots,v_{\subspace{n}_q}\in\F^n$ such that 
$\P \F^n=\{\subspace{v_1},\ldots,\subspace{v_{\subspace{n}_q}}\}$ and define
the matrix $S(n)=\big(v_1,\ldots,v_{\subspace{n}_q}\big)\in\F^{n\times\subspace{n}_q}$.
For a rank-metric code $\cC=\rs(G)$, where $G\in\F_{q^m}^{k\times n}$, 
the Hamming-metric code $\hat{\cC}=\rs(GS(n))\leq\F_{q^m}^{k\times\subspace{n}_q}$ 
is called the \textbf{projectivization} of~$\cC$.
\end{defi}

The matrix $S(n)$ is unique only up to monomial equivalence (that is, permutation and rescaling of its columns) and thus the same is true for $\hat{\cC}$.
Note that~$S(n)$ is the generator matrix of the $[(q^n-1)/(q-1),n]$-simplex code $\Sigma_n(q)$ over~$\F_q$, 
whereas the projectivization is  a Hamming-metric code over the field extension $\F_{q^m}$. 
Thus, if $\cC=\F_{q^m}^n$, i.e.,~$G$ is the identity matrix in
$\F_{q^m}^{n\times n}$, then $\hat{\cC}=\rs(GS(n))=\rs(S(n))$ is not the simplex code $\Sigma_n(q)$, but rather 
$\Sigma_n(q)\otimes\F_{q^m}$.

The following result has been proven in \cite[Thm.~6.15]{Ja23}. 
For completeness we include a short proof. 
Recall the notation from \cref{R-ClassMatr}.

\begin{theo}\label{T-ProjRepr}
If the \qM{} $\cM$ is s representable over~$\F_{q^m}$, then so is $\P\cM$. 
Precisely, let $G\in\F_{q^m}^{k\times n}$ and $\cM=\cM_G=(\F^n,\rho)$ as in \cref{T-ReprqMatr}.
Let $S(n)$ be as in \cref{D-ProjCode} and set $\hat{G}=GS(n)\in\F_{q^m}^{k\times\subspace{n}_q}$.
Consider the projectivization $\P\cM_G=(\P\F^n,r)$.
Then $r(\{\subspace{v_{i_1}},\ldots,\subspace{v_{i_t}}\})=\rk\big(\hat{G}_\Delta\big)$ for any $\Delta=\{i_1,\ldots,i_t\}\subseteq\cR_n$. 
In other words, $\P\cM_G$ is isomorphic to the vector matroid $M_{\hat{G}}$ and thus representable by the projectivization $\hat{\cC}=\rs(\hat{G})$. 
\end{theo}

\begin{proof}
\cref{D-ProjectqM} and \cref{T-ReprqMatr} imply 
\[
  r(\{\subspace{v_{i_1}},\ldots,\subspace{v_{i_t}}\})=\rho(\subspace{v_{i_1},\ldots,v_{i_t}})
  =\rk\big(G\big(v_{i_1},\ldots,v_{i_t}\big)\big)=\rk(\hat{G}_\Delta).
  \qedhere
\]
\end{proof}

We will see below in \cref{E-Converse} that representability of $\P\cM$ does not imply representability of~$\cM$.

\begin{prop}\label{P-IsoqMM}
Let $\cM=(E,\rho)$ and $\cM'=(E',\rho')$ be \qM{}s.
Then 
\[
   \cM\cong\cM'\Longrightarrow\P\cM\cong\P\cM'.
\]
\end{prop}

\begin{proof}
Let $\alpha:E\longrightarrow E'$ be an isomorphism such that $\rho'(\alpha(V))=\rho(V)$ for all $V\in\cL(E)$. 
Denote the rank functions of the projectivization matroids $\P\cM$ and $\P\cM'$ by~$r$ and~$r'$, respectively.
The map~$\alpha$ induces a well-defined bijection $\hat{\alpha}:\P E\longrightarrow\P E'$ via $\hat{\alpha}(\subspace{v})=\subspace{\alpha(v)}$.
Now we have for any subset $A=\{\subspace{v_1},\ldots,\subspace{v_t}\}\subseteq\P E$
\[
  r'(\hat{\alpha}(A))=r'(\{\subspace{\alpha(v_1)},\ldots,\subspace{\alpha(v_t)}\})=\rho'(\subspace{\alpha(v_1),\ldots,\alpha(v_t)})
  =\rho'(\alpha(\subspace{A}))=\rho(\subspace{A})=r(A).
\]
Hence $\P\cM$ and $\P\cM'$ are isomorphic. 
\end{proof}

The following example shows that the converse in \cref{P-IsoqMM} is not true.

\begin{exa}\label{E-Converse}
Recall that a $k$-spread of~$\F^n$ is a collection~$\cV$ of $k$-dimensional subspaces of~$\F^n$ such that 
$V\cap V'=0$ for all distinct $V,V'\in\cV$ and $\cup_{V\in\cV}V=\F^n$. 
A $k$-spread exists iff $k\mid n$, and in that case its size is $(q^n-1)/(q^k-1)$.
Any $k$-spread~$\cV$ gives rise to a \qM{} $\cM_{\cV}=(\F^n,\rho)$ via 
\[
  \rho(V)=\left\{\begin{array}{cl}k-1,&\text{if }V\in\cV,\\ \min\{k,\dim V\},&\text{otherwise}\end{array}\right.
\]
(see \cite[Prop.~4.6]{GLJ22Gen} for a more general version).
Let now $k=2$ and~$n=4$.
In \cite[Ex.~8.4]{GLJ24DSCyc} two 2-spreads $\cV=\{V_1,\ldots,V_{10}\}$ and $\cW=\{W_1,\ldots,W_{10}\}$ of~$\F_3^4$ are 
presented such that $\cM_{\cV}\not\cong\cM_{\cW}$. 
The latter implies that there is no isomorphism $\alpha:\F_3^4\longrightarrow\F_3^4$ such that $\alpha(V_i)\in\cW$ for all~$i$.
We will show now that the projectivization matroids are isomorphic.
Let $\P\cM_\cV=(\P\F_3^4,r)$ and $\P\cM_\cW=(\P\F_3^4,r')$.
Note that $|\P\F_3^4|=(3^4-1)/(3-1)=40$.
Every subspace~$V_i$ contains exactly 4 distinct lines, say $\subspace{v_{i,1}},\ldots,\subspace{v_{i,4}}$.
Thus $V_1,\ldots,V_{10}$ give rise to a partition $P=\big(P_i\big)_{i\in[10]}$ of $\P\F_3^4$, where $P_i=\{\subspace{v_{i,1}},\ldots,\subspace{v_{i,4}}\}$.
Likewise, the spread~$\cW$ induces a partition $P'=\big(P'_i\big)_{i\in[10]}$, where $P'_i=\{\subspace{w_{i,1}},\ldots,\subspace{w_{i,4}}\}$.
Choose any bijection $\beta:\P\F_3^4\longrightarrow\P\F_3^4$ such that $\beta(P_i)=P'_i$ for $i\in[10]$.
Clearly, any subset $A\subseteq\P\F_3^4$ satisfies the equivalence $A\subseteq P_i\iff \subspace{A}\leq V_i$.
This implies for any nonempty subset~$A$
\[
   r(A)=\rho(\subspace{A})=\left\{\begin{array}{cl}1,&\text{if }A\subseteq P_i\text{ for some i},\\ 
      \min\{2,\dim\subspace{A}\},&\text{otherwise.}\end{array}\right.
\]
The analogous property is true for the rank function~$r'$ and the partition~$P'$.
Noticing further that $\dim\subspace{A}=1\Longleftrightarrow|A|=1\Longleftrightarrow|\beta(A)|=1\Longleftrightarrow\dim\subspace{\beta(A)}=1$, we 
conclude  that $r'(\beta(A))=r(A)$ for all subsets $A\subseteq\P\F_3^4$.
Hence $\P\cM_\cV\cong\P\cM_\cW$.
All of this shows that the converse of \cref{P-IsoqMM} is not true.
Finally, in \cite[Ex.~8.4(3)]{GLJ24DSCyc} it is shown that $\cM_\cV$ is representable over~$\F_9$, 
whereas $\cM_\cW$ is not representable over any field extension of~$\F_3$.
Thanks to \cref{T-ProjRepr}, the projectivization $\P\cM_\cV$ is representable over~$\F_9$, 
and hence so is its isomorphic copy $\P\cM_\cW$.
This shows that the converse of \cref{T-ProjRepr} is not true.
\end{exa}

In the next section we will discuss  \qM{}s with isomorphic projectivizations in more detail. 

We now turn to the $t$-th higher rank-weight enumerator (resp.\ rank-support distribution) of~$\cC$.
As we will see, it is closely related to the $t$-th higher Hamming-weight enumerator (resp.\ Hamming-support distribution) of the 
projectivization~$\hat{\cC}$.
For $t=1$ the result in~(b) below has been proven in \cite[Thm.~4.9]{ABNR22} with the aid of $q$-systems. 
Recall \cref{R-HammCodes}.

\begin{theo}\label{T-HigherWeightsProj}
Define $\phi:\cL(\F^n)\longrightarrow 2^{\cR_n},\  Z\longmapsto \{j\in\cR_n\mid v_j\not\in Z^\perp\}$,
where~$Z^\perp$ refers to the standard dot product on~$\F^n$.
Let $\cC\leq\F_{q^m}^n$ and $\hat{\cC}\leq\F_{q^m}^{k\times\subspace{n}_q}$ be its projectivization.
\begin{alphalist}
\item Let $\cC'\leq_{q^m} \cC$ and $\hat{\cC'}\leq_{q^m}\hat{\cC}$ be its projectivization. Let $V\in\cL(\F^n)$. Then
        \begin{equation}\label{e-suppRH}
             \supp(\cC')=V\Longleftrightarrow\supp_{\text{H}}(\hat{\cC'})=\phi(V).        
        \end{equation}
       As a consequence, the $t$-th higher rank-support distribution of~$\cC$ and the $t$-th higher Hamming-support distribution of~$\hat{\cC}$ 
       determine each other. Precisely,
       \[
                |\{C'\leq_{q^m} \hat{\cC}\mid \dim C'=t,\, \supp_{\text{H}}(C')=\phi(V)\}|=|\{\cC'\leq_{q^m} \cC\mid \dim\cC'=t,\, \supp(\cC')=V\}|
       \]
       for all $t\in[k]_0$ and $V\in\cL(\F^n)$.     
\item The $t$-th higher rank-weight enumerator of~$\cC$ and the $t$-th higher Hamming-weight enumerator of $\hat{\cC}$
        determine each other. 
        Precisely, let $W^{(t)}_{\cC}=\sum_{i=0}^nA_i^{(t)}x^{n-i}y^i$ be as in \cref{D-Support} and $W_{\hat{\cC},\text{H}}^{(t)}$
        as in \cref{R-HammCodes}.
        Then 
        \begin{equation}\label{e-WW}
    W_{\hat{\cC},\text{H}}^{(t)}=\sum_{i=0}^n A_i^{(t)}x^{\subspace{n-i}_q}y^{\subspace{n}_q-\subspace{n-i}_q}
\end{equation}
for all $t\in[k]_0$. 
\end{alphalist}
\end{theo}

\begin{proof}
(a) The equivalence~\eqref{e-suppRH} has been proven in \cite[Lem.~6.19]{Ja23}. 
Since every subcode of~$\hat{\cC}$ is the projectivization of a subcode of~$\cC$, the rest follows.
\\
(b) Let $\cC'\leq_{q^m}\cC$ with $\supp(\cC')=V$ and suppose $\wtrk(\cC')=i$. Thus $\dim V=i$. 
Then $\dim V^\perp=n-i$ and therefore $|\phi(V)|=(q^n-1)/(q-1)-(q^{n-i}-1)/(q-1)=\subspace{n}_q-\subspace{n-i}_q$.
This establishes \eqref{e-WW}.
\end{proof}

\begin{rem}\label{R-MonSubWeightEnum}
\eqref{e-WW} can be expressd in terms of the monomial substitution $W_{\hat{\cC},H}^{(t)}=\Omega_{\cU}(W^{(t)}_{\cC})$, where 
$\cU=(u_{a,b})_{a,b\in\N_0}$ with
\[
   u_{a,b}=\left\{\begin{array}{cl} x^{\subspace{a}_q}y^{\subspace{n}_q-\subspace{a}_q},&\text{if $a\in[n]_0$ and }b=n-a,\\
                             0,&\text{otherwise.}\end{array}\right.
\]
Conversely, $W^{(t)}_{\cC}$ can be obtained from $W_{\hat{\cC},H}^{(t)}$ via the monomial substitution 
$x^{\subspace{a}_q}y^{\subspace{n}_q-\subspace{a}_q}\longmapsto x^ay^b$ for all $a\in[n]_0$ (no other monomials occur
in $W_{\hat{\cC},H}^{(t)}$).
\end{rem}

\cref{T-HigherWeightsProj} together with \cref{C-SubA} has an interesting consequence.
Consider the \qM{} $\cM_G$ and its projectivization $\P\cM_G$, where $G$ is as in \cref{T-HigherWeightsProj}.
From \cref{C-SubA} we know that the Whitney function of~$\cM_G$ and the collection of all higher rank-weight enumerators of $\cC=\rs(G)$ 
determine each other. 
The analogous result for matroids and Hamming-metric codes, established in  \cite[Thm.~5.11]{JuPe13} and   \cite[Thm.~4]{Bri10},
states that the Whitney function of the matroid represented by $GS(n)$ and the collection of all higher Hamming-weight enumerators of $\hat{\cC}=\rs(GS(n))$ determine each other.
Since the  matroid represented by $GS(n)$ is isomorphic to $\P\cM_G$ (see \cref{T-ProjRepr}), \cref{T-HigherWeightsProj} implies that the
Whitney functions of a representable \qM{}~$\cM_G$ and its projectivization $\P\cM_G$ determine each other via a monomial substitution.
The latter could be obtained explicitly by composing all monomial substitutions involved in the various steps.
This, however, becomes extremely cumbersome, and only applies to representable \qM{}s.
In the following we will derive the relation between the Whitney functions $R_{\cM}$ and $R_{\P\cM}$ directly. 
It will apply to all \qM{}s.
We need some preparation. 

Recall that we fix $n\in\N$ and an $n$-dimensional $\F$-vector space~$E$.

\begin{defi}\label{D-alpha}
For $i\in\N_0$ and $r\in[n]_0$ set 
$\gamma_i^r(n)=\big|\big\{A\subseteq\P E\,\big|\, |A|=i,\ \dim\subspace{A}=r\big\}\big|$.
\end{defi}

We have the following properties leading to an explicit expression for $\gamma_i^r(n)$.

\begin{prop}\label{P-alpha}
Fix $i\in\N_0$ and $r\in[n]_0$.
\begin{alphalist}
\item $\gamma_i^r(n)=0$ for $i<r$ or $i>\subspace{n}_q$. 
      Moreover, $\gamma_0^0(n)=1$ and $\gamma_n^n(n)=\big(\prod_{\ell=0}^{n-1}\frac{q^n-q^\ell}{q-1}\big)/n!$. 
\item $\gamma_i^r(n)=\Gaussian{n}{r}_q\gamma_i^r(r)$.
\item $\sum_{j=0}^n\Gaussian{n}{j}_q\gamma_i^j(j)=\Binom{\subspace{n}_q}{i}$.
\item $\gamma_i^r(r)=\sum_{j=0}^r\Gaussian{r}{j}_q(-1)^jq^{\BinomS{j}{2}}\Binom{\subspace{r-j}_q}{i}$.
\end{alphalist}
\end{prop}

A similar expression as in (d) also appears in \cite[Prop.~4.1]{LaTh94} for counting particular matrices over~$\F$.

\begin{proof}
(a) is clear.
\\
(b) Note first that there are $\Gaussian{n}{r}_q$ subspaces of~$E$ of dimension~$r$. 
For each such subspace~$W$ we have
\begin{align*}
   \gamma_i^r(r)&=\Big|\big\{A\in\P W\,\big|\, |A|=i,\, \dim\subspace{A}=r\big\}\big|
                      =\big|\big\{A\in\P W\,\big|\, |A|=i,\, \subspace{A}=W\big\}\big|\\
     &=\big|\big\{A\in\P E\,\big|\, |A|=i,\, \subspace{A}=W\big\}\big|,
\end{align*}
and this proves~(b).
\\
(c) Write $\P E=\{\subspace{v_1},\ldots,\subspace{v_{\subspace{n}_q}}\}$ and let $\cL_i(\cR_n)$be the set of $i$-subsets of~$\cR_n$.
Consider the map
\[
    \phi:\cL_i(\cR_n)\longrightarrow \cL(E),\ \{s_1,\ldots,s_i\}\longmapsto \subspace{v_{s_1},\ldots,v_{s_i}}.
\]
There are $\Gaussian{n}{j}_q$ subspaces~$W$ of dimension~$j$.
For each of them $|\phi^{-1}(W)|=\gamma_i^j(j)$, and hence
$|\cL_i(\cR_n)|=\sum_{j=0}^n\Gaussian{n}{j}_q\gamma_i^j(j)$. 
Now the statement follows from $|\cL_i(\cR_n)|=\Binom{\subspace{n}_q}{i}$.
\\
(d) Consider the maps
\[
  f:\cL(E)\longrightarrow\Z,\ W\longmapsto \Binom{\subspace{w}_q}{i},\qquad 
  g:\cL(E)\longrightarrow\Z,\ W\longmapsto \gamma_i^w(w),
\]
where again $w=\dim W$.
Then~(c) implies $f(W)=\sum_{X\leq W}g(X)$ for all $W\leq E$. 
Hence M\"obius inversion on the lattice $\cL(E)$ leads to
\begin{align*}
    \gamma_i^w(w)=g(W)&=\sum_{Z\leq W}\mu(Z,W)f(Z)=\sum_{Z\leq W}(-1)^{w-\dim Z}q^{\BinomS{w-\dim Z}{2}}f(Z)\\
        &=\sum_{z=0}^w\GaussianD{w}{z}_q(-1)^{w-z}q^{\BinomS{w-z}{2}}\Binom{\subspace{z}_q}{i}\\
        &=\sum_{j=0}^w\GaussianD{w}{j}_q(-1)^{j}q^{\BinomS{j}{2}}\Binom{\subspace{w-j}_q}{i}. 
        \qedhere
\end{align*}
\end{proof}

Now we are ready to discuss the Whitney functions of a \qM{} and its projectivization.

\begin{theo}\label{T-WhitneyProj}
Let $\cM=(E,\rho)$ be a \qM{} and $\P\cM=(\P E,r)$ be its projectivization matroid. 
Then the Whitney functions of $\cM$ and $\P\cM$ determine each other.
\end{theo}

An explicit relation between $R_\cM$ and $R_{\P\cM}$ will be given in \cref{C-SubWhitney}.

\begin{proof}
Let $k=\rho(E)$ be the rank of~$\cM$ and $R_\cM=\sum_{i=0}^k\sum_{j=0}^{n-k}\nu_{i,j}x^iy^j$ be as in \eqref{e-WhitExpand}.
Using \cref{D-alpha} we compute
\begin{align}
     R_{\P\cM}&=\sum_{A\subseteq\P E}x^{r(\P E)-r(A)}y^{|A|-r(A)}
              =\sum_{V\leq E}\sum_{A\subseteq\P V\atop \langle A\rangle=V}x^{k-\rho(V)}y^{\dim V-\rho(V)}y^{|A|-\dim V} \label{e-RPM2}\\
            &=\sum_{b=0}^n\sum_{V\leq E\atop \dim V=b}x^{k-\rho(V)}y^{b-\rho(V)}\sum_{\ell=b}^{\subspace{b}_q}\gamma_\ell^b(b)y^{\ell-b} \nonumber\\
            &=\sum_{b=0}^{n}\sum_{a=0}^k\nu_{k-a,b-a}x^{k-a}y^{b-a}\sum_{\ell=b}^{\subspace{b}_q}\gamma_\ell^b(b)y^{\ell-b}, \label{e-RPM3}
\end{align}
where the last step follows from  $\big|\{V\leq E\mid \dim V=b,\,\rho(V)=a\}\big|=\nu_{k-a,b-a}$; see \eqref{e-WhitCoeff}.
With the index change $i=k-a$ and $j=b-a$ we continue

\begin{align}
   R_{\P\cM}&=\sum_{j=0}^{n-k}\sum_{i=0}^{k}\nu_{i,j}x^iy^j
                        \sum_{\ell=j+k-i}^{\subspace{j+k-i}_q}\gamma_{\ell}^{j+k-i}(j+k-i)y^{\ell-j+i-k} \nonumber\\
                    &=\sum_{i=0}^{k}\sum_{j=0}^{n-k}\nu_{i,j}x^i
                       \sum_{\ell=j+k-i}^{\subspace{j+k-i}_q}\gamma_{\ell}^{j+k-i}(j+k-i)
                        y^{\ell+i-k}\label{e-RPM1}\\
                 &=\sum_{i=0}^{k}\sum_{\ell=k-i}^{\subspace{n}_q}\sum_{j=0}^{\min\{n-k,\ell-k+i\}}
                                   \nu_{i,j}\gamma_{\ell}^{j+k-i}(j+k-i) x^i y^{\ell+i-k} \nonumber\\
                 &=\sum_{i=0}^{k}\sum_{\ell=0}^{\subspace{n}_q-k+i}\,\sum_{j=0}^{\min\{n-k,\ell\}}
                                   \nu_{i,j}\gamma_{\ell+k-i}^{j+k-i}(j+k-i)x^iy^{\ell}   \nonumber\\
                     & =\sum_{i=0}^{k}\sum_{\ell=0}^{\subspace{n}_q-k+i}\mu_{i,\ell}x^i y^{\ell}, \nonumber
\end{align}
where 
\begin{equation}\label{e-munu}
   \mu_{i,\ell}=\sum_{j=0}^{\min\{n-k,\ell\}}\nu_{i,j}\gamma_{\ell+k-i}^{j+k-i}(j+k-i)\ \text{ for }\ i\in[k]_0\text{ and }
   \ell\in[\subspace{n}_q-k+i]_0.
\end{equation}
From  \eqref{e-WhitneyMatr} we know that $\mu_{i,\ell}=0$ for $\ell>\subspace{n}_q-k$, and thus 
we have to consider \eqref{e-munu} only for $\ell\in[\subspace{n}_q-k]_0$.
For any  $i\in[k]_0$ the identities in \eqref{e-munu} form the linear equation

\begin{equation}\label{e-LinSys}
  \begin{pmatrix}\mu_{i,0}\\[.6ex] \mu_{i,1}\\[.6ex] \vdots\\[.6ex] \mu_{i,n-k}\\[.6ex] \hline 
     \mu_{i,n-k+1}\\[.6ex] \vdots\\[.6ex] \mu_{i,\subspace{n}_q-k}\end{pmatrix}=
  \begin{pmatrix}\gamma_{k-i}^{k-i}(k-i)&0 &\cdots &0\\[.4ex] \gamma_{k-i+1}^{k-i}(k-i)&
              \gamma_{k-i+1}^{k-i+1}(k-i+1)&\cdots &0\\[.4ex]
      \vdots &\vdots &\ddots  &\vdots\\[.4ex]
      \gamma_{n-i}^{k-i}(k-i)&\gamma_{n-i}^{k-i+1}(k-i+1)&\ldots & \gamma_{n-i}^{n-i}(n-i)\\[.4ex] \hline
      \gamma_{n-i+1}^{k-i}(k-i)&\gamma_{n-i+1}^{k-i+1}(k-i+1)&\ldots & \gamma_{n-i+1}^{n-i}(n-i)\\[.4ex]
      \vdots &\vdots &  &\vdots\\[.4ex]
      \gamma_{\subspace{n}_q-i}^{k-i}(k-i)&\gamma_{\subspace{n}_q-i}^{k-i+1}(k-i+1)&\ldots & \gamma_{\subspace{n}_q-i}^{n-i}(n-i)
      \end{pmatrix}
      \begin{pmatrix}\nu_{i,0}\\[.5ex] \nu_{i,1}\\[.5ex] \vdots\\[.5ex] \nu_{i,n-k}\end{pmatrix}.
\end{equation}
This shows that for every~$i$ the vector $(\nu_{i,0},\, \ldots,\, \nu_{i,n-k})$ fully determines 
$(\mu_{i,0},\, \mu_{i,1},\, \ldots,\, \mu_{i,\subspace{n}_q-k})$. 
Conversely, note that the upper block of the big matrix has size $(n-k+1)\times(n-k+1)$ and is invertible since 
$\gamma_t^t(t)\neq0$ for every $t$ thanks to \cref{P-alpha}(a).
Thus every vector $(\mu_{i,0},\, \ldots,\, \mu_{i,\subspace{n}_q-k})\T$ in the column space of the big matrix 
uniquely determines the vector $(\nu_{i,0},\, \ldots,\, \nu_{i,n-k})$.
All of this shows that $R_\cM$ and $R_{\P\cM}$ determine each other.
\end{proof}

\begin{rem}\label{R-WhitProjLower}
The invertibility of the upper block of the big matrix in \eqref{e-LinSys} shows that for every~$i$ the coefficients 
$(\mu_{i,n-k+1},\, \ldots,\, \mu_{i,\subspace{n}_q-k})$ are determined by 
$(\mu_{i,0},\, \mu_{i,1},\, \ldots,\, \mu_{i,n-k})$.
In other words, the Whitney function $R_{\P\cM}=\sum_{i=0}^{k}\sum_{\ell=0}^{\subspace{n}_q-k}\mu_{i,\ell}x^i y^{\ell}$
is fully determined by its truncation $\sum_{i=0}^{k}\sum_{\ell=0}^{n-k}\mu_{i,\ell}x^iy^{\ell}$.
\end{rem}

The relation between $R_\cM$ and $R_{\P\cM}$ can be expressed in form of monomial substitutions.

\begin{cor}\label{C-SubWhitney}
Let $\cM=(E,\rho)$ be a \qM{} of rank~$k$.
Denote the inverse of the upper block in \eqref{e-LinSys} by 
$\big(\beta^{(i)}_{a,b}\big)_{a=0,\ldots,n-k\atop b=0,\ldots,n-k}$. 
Define  $\cG=(g_{i,j})_{i,j\in\N_0}$ and $\cH=(h_{i,j})_{i,j\in\N_0}$, where 
\[
     g_{i,j}=x^i\sum_{\ell=j+k-i}^{\subspace{j+k-i}_q}\gamma_{\ell}^{j+k-i}(j+k-i)y^{\ell+i-k},\quad
     h_{i,j}=\left\{\begin{array}{cl} x^i\sum_{\ell=j}^{n-k}\beta^{(i)}_{\ell,j}y^\ell,&\text{if }j\leq n-k,\\[.5ex] 0,&\text{otherwise.}\end{array}\right.
\]
Then $R_{\P\cM}=\Omega_{\cG}(R_\cM)$ and $R_{\cM}=\Omega_{\cH}(R_{\P\cM})$.
\end{cor}

\begin{proof}
Let $R_{\cM}=\sum_{i=0}^k\sum_{\ell=0}^{n-k}\nu_{i,\ell}x^iy^\ell$ and 
$R_{\P\cM}=\sum_{i=0}^{k}\sum_{j=0}^{\subspace{n}_q-k}\mu_{i,j}x^iy^{j}$.
The first identity is in \eqref{e-RPM1}.
As for the second one, note that the upper block in \eqref{e-LinSys} is lower triangular and thus 
$\nu_{i,\ell}=\sum_{j=0}^\ell\beta^{(i)}_{\ell,j}\mu_{i,j}$ for $\ell\in[n-k]_0$.
As a consequence, 
\[
   R_\cM
   =\sum_{i=0}^{k}\sum_{\ell=0}^{n-k}\sum_{j=0}^\ell\beta^{(i)}_{\ell,j}\mu_{i,j}x^iy^\ell
   =\sum_{i=0}^{k}\sum_{j=0}^{n-k}\mu_{i,j}\sum_{\ell=j}^{n-k}\beta^{(i)}_{\ell,j}x^iy^\ell
   =\sum_{i=0}^{k}\sum_{j=0}^{\subspace{n}_q-k}\mu_{i,j} h_{i,j}.
   \qedhere
\]
\end{proof}

\begin{exa}\label{E-UnifSimplex}
Let $\cU_{n,n}=(\F^n,\rho)$ be the free \qM{} of rank~$n$. Thus $\rho(V)=\dim V$ for all $V\leq \F^n$. 
Its Whitney function is
\[
   R_{\cU_{n,n}}=\sum_{V\leq\F^n}x^{n-\dim V}y^{\dim V-\rho(V)}=\sum_{i=0}^n\GaussianD{n}{i}_qx^i.
\]
Set~$M:=\P \cU_{n,n}$.
With the aid of \cref{C-SubWhitney} we compute the Whitney function of $M$ as
\begin{align}
  R_M&=\sum_{i=0}^n\GaussianD{n}{i}_qg_{i,0}
  =\sum_{i=0}^n\GaussianD{n}{i}_qx^i\sum_{\ell=n-i}^{\subspace{n-i}_q}\gamma_\ell^{n-i}(n-i)y^{\ell+i-n} \nonumber\\
   &=\sum_{i=0}^n\GaussianD{n}{i}_q\sum_{\ell=n-i}^{\subspace{n-i}_q}\gamma_\ell^{n-i}(n-i)x^iy^{\ell+i-n}.\label{e-WhitProjUnif}
\end{align}
The expression for $R_M$ in \eqref{e-WhitProjUnif} can also be seen directly as follows.
Since $\cU_{n,n}$ is the \qM{} represented by the matrix $G=I_n$, \cref{T-ProjRepr} implies that the projectivization~$M$ 
is isomorphic to the matroid with ground set $\cR_n$ represented by the matrix $S(n)$ defined in the same theorem.
Recall that $S(n)$ is a generator matrix for the Simplex code of dimension~$n$ and length $\subspace{n}_q$. 
The definition in \eqref{e-WhitneyMatr} leads to
\begin{equation}\label{e-RMbij}
    R_{M}=\sum_{i=0}^n\sum_{j=0}^{\subspace{n}_q-n}b_{i,j}x^iy^j,\ \text{ where }\ b_{i,j}=|\{A\subseteq\cR_n\mid \rk S(n)_A=n-i,\,|A|=n-i+j\}|,
\end{equation}
and where $S(n)_A$ is the submatrix of~$S(n)$ consisting of the columns indexed by~$A$.
Note that 
$b_{i,j}=\big|\big\{A\subseteq\P\F^n\,\big|\, \dim\subspace{A}=n-i,\, |A|=n-i+j\big\}\big|=\gamma_{n-i+j}^{n-i}(n)
=\Gaussian{n}{i}_q\gamma_{n-i+j}^{n-i}(n-i)$ by \cref{D-alpha} and \cref{P-alpha}(b).
Thanks to $\subspace{n-i}_q\leq\subspace{n}_q-i$ and the fact that $\gamma_\ell^{n-i}(n-i)=0$ for $\ell>\subspace{n-i}_q$, \eqref{e-RMbij}
equals \eqref{e-WhitProjUnif}.
Let us also determine the characteristic polynomial of~$M$. 
On the one hand, 
\begin{align}
    \chi_M(x)&=(-1)^nR_M(-x,-1)=(-1)^n\sum_{i=0}^n\GaussianD{n}{i}_q\sum_{\ell=n-i}^{\subspace{n-i}_q}\gamma_\ell^{n-i}(n-i)(-1)^{\ell-n}x^i
          \nonumber\\
        &=\sum_{i=0}^n\GaussianD{n}{i}_q\Big(\sum_{\ell=n-i}^{\subspace{n-i}_q}(-1)^\ell\gamma_\ell^{n-i}(n-i)\Big) x^i\label{e-chiM}.
\end{align}
On the other hand, \cite[Thm.~5.8]{Ja23} tells us that $\chi_M=\chi_{\cU_{n,n}}$. 
By definition (see \cref{R-CharPoly}), the latter is given by 
\begin{equation}\label{e-chiU}
  \chi_{\cU_{n,n}}=\sum_{b=0}^n\sum_{\dim V=b}(-1)^b q^{\BinomS{b}{2}}x^{n-b}=\sum_{i=0}^n(-1)^{n-i}\GaussianD{n}{i}_q q^{\BinomS{n-i}{2}}x^i.
\end{equation}
(this appears also in \cite[Ex.~5.51]{JuPe13}). Comparing \eqref{e-chiM} and \eqref{e-chiU} leads to the identities
\[
   \sum_{\ell=i}^{\subspace{i}_q}(-1)^{\ell-i}\gamma_\ell^i(i)=q^{\BinomS{i}{2}}
\]
for all $i\in\N_0$.  
\end{exa} 

We close this section with the following simple corollary.

\begin{cor}\label{C-WeakIsom}
Let $\cM_i,\, i=1,2,$ be \qM{}s such that~$\cM_i$ is representable  by the matrix $G_i\in\F_{q^m}^{k\times n}$.
Suppose $\P\cM_1\cong\P\cM_2$.
Then the rank-metric codes
$\cC_i=\rs(G_i),\,i=1,2,$ share all higher rank-weight enumerators.
\end{cor}

\begin{proof}
The assumption implies $R_{\P\cM_1}=R_{\P\cM_2}$ and thus $R_{\cM_1}=R_{\cM_2}$ thanks to \cref{C-SubWhitney}.
Now the statement  follows from \cref{C-SubA}.
\end{proof}

\section{Weak Isomorphisms}\label{S-WeakIso}

In this section we will discuss \qM{}s whose projectivizations are isomorphic. 
It will turn out that this is equivalent to the \qM{}s being weakly isomorphic in the sense of \cref{D-WeakIso} below.
We need the following facts for which we refer to \cite[Sec.~3]{BCIJS23}.

\begin{rem}\label{R-FlatCycFlatsLattice}
Let $\cM=(E,\rho)$ be a \qM{}. 
The \textbf{closure} of~$V\in\cL(E)$ is defined as the subspace $\cl(V)=\{x\in E\mid \rho(V+\subspace{x})=\rho(V)\}$, and~$V$ is called a \textbf{flat} if~$\cl(V)=V$.
The collection of all flats is denoted by $\cF(\cM)$, and $(\cF(\cM),\leq,\cap,\vee)$, where $V_1\vee V_2=\cl(V_1+V_2)$, is a geometric lattice 
with height function equal to $\rho$.
The inverse image of a flat $F\in\cF(\cM)$ under the closure operator is given by $\cl^{-1}(F)=\{V\in\cL(E)\mid \cl(V)=F\}=\{V\in\cL(F)\mid \rho(V)=\rho(F)\}$.
The notation $\cF(\,\cdot\,)$ and $\cl(\,\cdot\,)$ for the flats and the closure operator will also be used for matroids.
\end{rem}

\begin{defi}\label{D-WeakIso}
The \qM{}s $\cM_1=(E_1,\rho_1)$ and $\cM_2=(E_2,\rho_2)$ are called \textbf{weakly isomorphic} if there exists a dimension-preserving lattice isomorphism  $\alpha:\cF(\cM_1)\longrightarrow\cF(\cM_2)$, that is $\dim F=\dim\alpha(F)$ for all $F\in\cF(\cM_1)$.
\end{defi}

Note that since the height in the lattice of flats of a \qM{} equals the rank in the \qM{}, any lattice isomorphism between the lattices of flats of two \qM{}s 
is rank-preserving.
Thus, weakly isomorphic \qM{}s admit a dimension- and rank-preserving lattice isomorphism between their lattices of flats.

Clearly, isomorphic \qM{}s are weakly isomorphic. The converse is not true.

\begin{exa}\label{E-Converse2}
Consider the non-isomorphic \qM{}s $\cM_\cV$ and $\cM_\cW$ from \cref{E-Converse}. In that case $\cF(\cM_{\cV})=\cV\cup\{0,\F_3^4\}$ and 
$\cF(\cM_{\cW})=\cW\cup\{0,\F_3^4\}$; see also \cite[Ex.~8.4]{GLJ24DSCyc}. 
Therefore $\cF(\cM_\cV)$ and $\cF(\cM_\cW)$ admit a dimension-preserving lattice isomorphism.
Hence $\cM_\cV$ and~$\cM_\cW$ are weakly isomorphic, but not isomorphic.
\end{exa}

For matroids, the corresponding notion of weakly isomorphic coincides with isomorphic.
This is shown in \cref{T-FIsoIso} below. 
We believe that this is a well-known fact in matroid theory but could not find a source. 
For simple matroids this is mentioned at various places, for instance \cite[Bottom of p.~54]{Ox11}.

\begin{theo}\label{T-FIsoIso}
Let $M_1=(S_1,r_1),\,M_2=(S_2,r_2)$ be matroids and $\alpha:\cF(M_1)\longrightarrow\cF(M_2)$ be a cardinality-preserving lattice isomorphism.
Then $M_1\cong M_2$.
\end{theo}

\begin{proof}
Since $S_i\in\cF(M_i)$ we have $|S_1|=|S_2|$ and $r_1(S_1)=r_2(S_2)=:k$.
Let $F_1,\ldots,F_t$ be the flats of~$M_1$ of rank~$1$.
Then $\alpha(F_1),\ldots,\alpha(F_t)$ are the flats of~$M_2$ of rank~$1$.
The flat axioms imply $\bigcup_{j=1}^t F_j=S_1$ and likewise $\bigcup_{j=1}^t \alpha(F_j)=S_2$.
Furthermore, $F_j\cap F_{j'}=\cl_1(\emptyset)$ and $\alpha(F_j)\cap\alpha(F_{j'})=\cl_2(\emptyset)=\alpha(\cl_1(\emptyset))$ for all $j\neq j'$
(where $\cl_i(\cdot)$ denotes the closure in~$M_i$).
By assumption $|F_j|=|\alpha(F_j)|$ for all $j\in[t]$.
All of this implies the existence of a bijection 
\[
     \beta:S_1\longrightarrow S_2\ \text{ such that $\beta(\cl_1(\emptyset))=\cl_2(\emptyset)$ and $\beta(F_j)=\alpha(F_j)$ for $j\in[t]$.}
\]
Note that $\beta$ is not unique.
\\
1) We show that $\beta(F)=\alpha(F)$ for all $F\in\cF(M_1)$. 
By construction this is true for the flats of rank at most~$1$.
Choose $F\in\cF(M_1)$ and let $F_{i_1},\ldots,F_{i_\ell}$ be the flats of rank~$1$ contained in~$F$. 
Then $F=\bigvee_{j=1}^\ell F_{i_j}$ (see also \cite[Thm.~1.7.5]{Ox11}). 
But then we even have $F= \bigcup_{j=1}^\ell F_{i_j}$.
Indeed, for every $s\in F$ we have $\cl_1(\{s\})\subseteq F$ and either $\cl_1(\{s\})=\cl_1(\emptyset)$ or $\cl_1(\{s\})$ has rank~$1$.
In either case $s\in \bigcup_{j=1}^\ell F_{i_j}$ and thus  $F= \bigcup_{j=1}^\ell F_{i_j}$.
Similarly we have $\bigvee_{j=1}^\ell \alpha(F_{i_j})=\bigcup_{j=1}^\ell \alpha(F_{i_j})$.
All of this gives us $\beta(F)= \bigcup_{j=1}^\ell \beta(F_{i_j})=\bigcup_{j=1}^\ell \alpha(F_{i_j})=\bigvee_{j=1}^\ell \alpha(F_{i_j})
=\alpha(\bigvee_{j=1}^\ell F_{i_j})=\alpha(F)$, where the fourth step follows from the fact that~$\alpha$ is a lattice isomorphism.
\\
2) We show that $\cl_2(\beta(V))=\alpha(\cl_1(V))$ for all subsets $V\subseteq S_1$. Let $\cl_1(V)=F$ and $\cl_2(\beta(V))=G$.
Since $\beta(V)\subseteq\beta(F)=\alpha(F)$, which is a flat, we have $G\subseteq\alpha(F)$.
In the same way, $V=\beta^{-1}(\beta(V))\subseteq\beta^{-1}(G)=\alpha^{-1}(G)$ and thus $F=\cl_1(V)\subseteq\alpha^{-1}(G)$ because
the latter is a flat. Thus $\alpha(F)\subseteq G$ and all of this implies $\cl_2(\beta(V))=G=\alpha(F)=\alpha(\cl_1(V))$, as desired.
\\
3) Let $V\subseteq S_1$. Using that $\alpha$ is rank-preserving (because it is height-preserving), we obtain 
$r_1(V)=r_1(\cl_1(V))=r_2(\alpha(\cl_1(V)))=r_2(\cl_2(\beta(V)))=r_2(\beta(V))$.
Hence~$\beta$ is an isomorphism between the matroids ~$M_1$ and~$M_2$.
\end{proof}

We return to \qM{}s and show that \qM{}s are weakly isomorphic iff their projectivizations are isomorphic.
To do so, we need the following fact.

\begin{lemma}\label{P-MPMFlats}
Let $\cM$ be a \qM{} and $\P\cM$ be its projectivization. Then 
\[
   \gamma:\cF(\cM)\longrightarrow\cF(\P\cM),\quad F\longmapsto\P F
\]
is a bijection satisfying $F\leq F'\Longleftrightarrow \gamma(F)\subseteq \gamma(F')$ for all $F,F'\in\cF(\cM)$.
In particular, $\gamma$ is a lattice isomorphism. 
Furthermore, $|\gamma(F)|=(q^{\dim F}-1)/(q-1)$ for all $F\in\cF(\cM)$.
\end{lemma}

\begin{proof}
The fact that~$\gamma$ is a bijection and $\gamma$ and $\gamma^{-1}$ are order-preserving can be found in \cite[Lem.~16, Prop.~21]{JPV23} (see also \cite[Lem.~3.2]{Ja23}). Thus~$\gamma$ is a lattice isomorphism; see also 
\cite[Lem.~6.2]{GLJ25CF}.
The last part about $|\gamma(F)|$ is clear.
\end{proof}

Now we are ready to prove the desired equivalence.

\begin{theo}\label{T-PMIsoWeaklyIso}
Let $\cM_1$ and $\cM_2$ be \qM{}s. Then 
\[
   \P\cM_1\cong\P\cM_2\Longleftrightarrow \text{$\cM_1$ and $\cM_2$ are weakly isomorphic.}
\]
As a consequence, weakly isomorphic \qM{}s share the same Whitney function. 

\end{theo}

\begin{proof}
Thanks to \cref{T-FIsoIso} $\P\cM_1\cong\P\cM_2$ is equivalent to the existence of a 
cardinality-preserving lattice isomorphism between $\cF(\P\cM_1)$ and $\cF(\P\cM_2)$.
By \cref{P-MPMFlats} this is equivalent to the existence of 
a dimension-preserving lattice isomorphism between~$\cF(\cM_1)$ and~$\cF(\cM_2)$ and thus to~$\cM_1$ and~$\cM_2$ being weakly isomorphic.
The consequence follows from  \cref{T-WhitneyProj}.
\end{proof}

We derive another property of weakly isomorphic \qM{}s.
The following notion will be needed.

\begin{defi}\label{D-drBij}
Let $\cM_i=(E_i,\rho_i),\,i=1,2$, be \qM{}s and fix subsets $\cV_i\subseteq\cL(E_i)$. A map $\alpha:\cV_1\longrightarrow\cV_2$ is called a 
\textbf{dr-map} if~$\alpha$ is dimension- and rank-preserving, that is, $\dim V=\dim\alpha(V)$ and $\rho_1(V)=\rho_2(\alpha(V))$ for 
all $V\in\cV_1$.  
If in addition,~$\alpha$ is a bijection, we call it a \textbf{dr-bijection}.
\end{defi}

The notion of a dr-bijection depends on the underlying \qM{}s. The latter will always be clear from the context.
Now we obtain the following result. 
Note that it implies that weakly isomorphic \qM{}s have the same Whitney function -- as we know already from \cref{T-PMIsoWeaklyIso}.

\begin{theo}\label{T-WeakIsoDRBij}
Let $\cM_1=(E_1,\rho_1)$ and $\cM_2=(E_2,\rho_2)$ be weakly isomorphic \qM{}s and $\alpha:\cF(\cM_1)\longrightarrow\cF(\cM_2)$ a dimension-pre\-serving lattice isomorphism.
Then $\alpha$ extends to a dr-bijection $\beta:\cL(E_1)\longrightarrow\cL(E_2)$ such that $\beta(\cl_1^{-1}(F))=\cl_2^{-1}(\alpha(F))$ for all $F\in\cF(\cM_1)$. 
\end{theo}

As we know from \cref{E-Converse2} there exist non-isomorphic \qM{}s that admit a 
dimension-preserving lattice isomorphism between their lattices of flats.
This implies that the bijection~$\beta$ in \cref{T-WeakIsoDRBij} is in general not order-preserving.
Indeed, if~$\beta$ was order-preserving it would be a lattice isomorphism and thus be induced by a semi-linear isomorphism 
between~$E_1$ and~$E_2$ thanks to the Fundamental Theorem of Projective Geometry.
Thus,  the \qM{}s would be isomorphic if the ground field is a prime field.

\begin{proof}
Since $E_i\in\cF(\cM_i)$, we have $\dim E_1=\dim E_2$ and $\rho_1(E_1)=\rho_2(E_2)=:k$.
For $i=1,2$ and $r\in[k]_0$ define
\[
  \cL_r(E_i)=\{V\leq E_i\mid \rho_i(V)=r\}\ \text{ and }\ \cL_{\leq r}(E_i)=\{V\leq E_i\mid \rho_i(V)\leq r\}.
\]
In the same way, define $\cF_r(\cM_i)$ and $\cF_{\leq r}(\cM_i)$.
It suffices to show that for every~$r$ there exists a dr-bijection $\beta_r:\cL_r(E_1)\longrightarrow\cL_r(E_2)$ such that 
\begin{equation}\label{e-betar}
   \beta_r(F)=\alpha(F)\ \text{ and }\
   \beta_r(\cl_1^{-1}(F))=\cl_2^{-1}(\alpha(F)) \ \text{ for all $F\in\cF_r(\cM_1)$.}
\end{equation}
We induct on~$r$.
\\
1) Let $r=0$. Then $\cF_0(\cM_i)=\{\cl_i(0)\}$ and $\cL_0(E_i)=\cL(\cl_i(0))$.
By assumption $\dim\cl_1(0)=\dim\cl_2(0)$, and hence the subspaces in $\cL(\cl_1(0))$ and $\cL(\cl_2(0))$ are in a dimension-preserving one-to-one-correspondence.
In other words,  there exists a dr-bijection $\beta_0:\cL_0(E_1)\longrightarrow\cL_0(E_2)$ satisfying~\eqref{e-betar}.
\\
2) Suppose that for all $r\leq s$ there exists a dr-bijection~$\beta_r$ satisfying~\eqref{e-betar}. 
Set $\cF_{s+1}(\cM_1)=\{F_1,\ldots,F_t\}$.
Then $\cF_{s+1}(\cM_2)=\{\alpha(F_1),\ldots,\alpha(F_t)\}$ and
\begin{equation}\label{e-Disjoint}
   \bigcup_{j=1}^t\cl_1^{-1}(F_j)
   =\cL_{s+1}(E_1)\ \text{ and }
   \bigcup_{j=1}^t\cl_2^{-1}(\alpha(F_j))
   =\cL_{s+1}(E_2)\ 
   \text{ and the unions are disjoint.}
\end{equation}
We will show that for every $j\in[t]$ there exists a dr-bijection
 $\gamma_j:\cl_1^{-1}(F_j)\longrightarrow \cl_2^{-1}(\alpha(F_j))$.
 To do so, fix $j\in[t]$ and let $\cL(F_j)\cap\cF_{\leq s}(\cM_1)=\{G_1,\ldots,G_a\}$ (these are the flats in the sublattice strictly below~$F_j$).
Since~$\alpha$ is a lattice isomorphism, we obtain
\begin{equation}\label{e-CLF}
  \cL(F_j)=\cl_1^{-1}(F_j)\cup\bigcup_{\ell=1}^a\cl_1^{-1}(G_\ell)\ \text{ and }
  \cL(\alpha(F_j))=\cl_2^{-1}(\alpha(F_j))\cup\bigcup_{\ell=1}^a\cl_2^{-1}(\alpha(G_\ell)),
\end{equation}
where all unions are disjoint.
Using suitable restrictions of the maps $\beta_r$ for $r\leq s$ we obtain a dr-bijection
\[
    \hat{\beta}:\bigcup_{\ell=1}^a\cl_1^{-1}(G_\ell)\longrightarrow\bigcup_{\ell=1}^a\cl_2^{-1}(\alpha(G_\ell)).
\]
Since $\dim F_j=\dim\alpha(F_j)$, the subspaces in $\cL(F_j)$ and $\cL(\alpha(F_j))$ are in a one-to-one dimension-preserving correspondence,
and therefore~\eqref{e-CLF} implies the existence of a dimension-preserving bijection
$\gamma_j:\cl_1^{-1}(F_j)\longrightarrow \cl_2^{-1}(\alpha(F_j))$. It is trivially rank-preserving.
Moreover, its dimension-preserving property implies that $\gamma_j(F_j)=\alpha(F_j)$.
Using~\eqref{e-Disjoint} and combining all dr-bijections $\gamma_1,\ldots,\gamma_t$ we arrive at the desired dr-bijection $\beta_{s+1}$.
\end{proof}

The following examples illustrate that for the results above both properties of the map~$\alpha$ in \cref{D-WeakIso} are crucial, 
being a lattice isomorphism and being dimension-preserving.

\begin{exa}\label{E-LatticeIsoNotDR}
A simple example of \qM{}s where the lattices of flats are isomorphic but do not admit a dimension-preserving isomorphism is given by the \qM{}s $\mathsf{U}_{1,3},\,{\mathsf P}_1$, and ${\mathsf P}_2$ in \cite[p.~325 and 326]{CeJu24}. 
In each case the sets of flats is $\{\cl(0),E\}$ and thus the lattices are isomorphic. However, since $\cl(0)$ takes dimensions $0,\,1$, and~$2$, respectively, the lattices of flats do not admit a dimension-preserving isomorphism.
Hence the \qM{}s are not weakly isomorphic. 
One easily checks that they have different Whitney functions and hence there do not exist dr-bijections between the subspace lattices 
of their ground spaces (compare with \cref{T-WeakIsoDRBij}).
\end{exa}

A more interesting example is the following, where $\cl(0)=0$ for both \qM{}s.

\begin{exa}\label{E-LatticeIsoNotDR2}
Let~$E$ be an $n$-dimensional $\F$-vector space and $N\in\N$. Suppose we have a collection of subspaces
$\cF=\{0,F_1,\ldots,F_N,E\}$ that form a lattice of the form
\begin{equation}\label{e-Lattice}
\begin{array}{l}
   \begin{xy}
   (45,5)*+{0}="a0";%
   (10,15)*+{\mbox{\footnotesize{$F_1$}}}="b0";%
   (20,15)*+{\mbox{\footnotesize{$F_2$}}}="b1";%
   (30,15)*+{\mbox{\footnotesize{$F_3$}}}="b2";%
   (40,15)*+{\mbox{\footnotesize{$\cdots$}}}="b4";%
   (45,15)*+{\mbox{\footnotesize{$\cdots$}}}="b5";%
   (50,15)*+{\mbox{\footnotesize{$\cdots$}}}="b6";%
   (60,15)*+{\mbox{\footnotesize{$F_{N-2}$}}}="b7";%
   (70,15)*+{\mbox{\footnotesize{$F_{N-1}$}}}="b8";%
   (80,15)*+{\mbox{\footnotesize{$F_{N}$}}}="b9";%
   (45,25)*+{\mbox{\footnotesize{$E$}}}="c";%
    {\ar@{-}"a0";"b0"};
    {\ar@{-}"a0";"b1"};
     {\ar@{-}"a0";"b2"};
      {\ar@{-}"a0";"b7"};
      {\ar@{-}"a0";"b8"};
      {\ar@{-}"a0";"b9"};
     {\ar@{-}"b0";"c"};
    {\ar@{-}"b1";"c"};
     {\ar@{-}"b2";"c"};
      {\ar@{-}"b7";"c"};
      {\ar@{-}"b8";"c"};
      {\ar@{-}"b9";"c"};
  \end{xy}
\end{array}
\end{equation}
Then~$\cF$ satisfies the axioms of flats of a \qM{} (see \cite[Def.~10]{BCJ22}) iff $\hat{\cF}:=\{F_1,\ldots,F_N\}$ is a \emph{mixed spread}
of~$E$, that is, $F_i\cap F_j=\{0\}$  for $i\neq j$ and $\bigcup_{i=1}^{N}F_i=E$,
and the subspaces~$F_i$ may have different dimensions.
In other words, every nonzero vector of~$E$ is in exactly one subspace~$F_i$.
It is not hard to give distinct instances of~$\hat{\cF}$ for $E=\F_2^6$ and $N=15$.
Start with a 3-spread~$\cG=\{V_1,\ldots,V_9\}$, that is, $\dim V_i=3$ for $i\in[9]$.
Note that each~$V_i$ is the union of 7 distinct lines.
Without loss of generality we may assume that $V_1=\subspace{e_1,e_2,e_3},\,V_2=\subspace{e_4,e_5,e_6},\,
V_3=\subspace{e_1+e_4,e_2+e_5,e_3+e_6}$.
Dividing~$V_1$ into its collection of~$7$ lines we obtain the mixed spread
\[  
   \hat{\cF}_1=\{\subspace{v}\mid v\in V_1\setminus\{0\}\}\cup\{V_2,\ldots,V_9\}.
\]
Furthermore, a suitable grouping of the~$21$ nonzero vectors in $V_1\cup V_2\cup V_3$ leads to the mixed spread
\begin{align*}
   \hat{\cF}_2&=\{\subspace{e_1+e_3},\,\subspace{e_4+e_6},\,\subspace{e_1+e_3+e_4+e_6}\}\\
        &\hspace*{1em}\cup\{\subspace{e_1,e_2},\subspace{e_4,e_5},\subspace{e_3,e_6},\subspace{e_1+e_4,e_2+e_5},
          \subspace{e_2+e_3,e_5+e_6},\subspace{e_1+e_2+e_3,e_4+e_5+e_6}\}\\
        &\hspace*{1em}\cup\{V_4,\ldots,V_9\}.
\end{align*}
Thus we have two mixed spreads of size~$15$ and therefore $\cF_i=\hat{\cF}_i\cup\{0,\F_2^6\},\,i=1,2$, 
form lattices as in \eqref{e-Lattice}.
Hence we obtain two \qM{}s $\cM_1$ and $\cM_2$ on the ground space~$\F_2^6$ whose lattices of flats are 
isomorphic but do not admit a dimension-preserving lattice isomorphism.
Their Whitney functions are
\begin{align*}
   R_{\cM_1}&=x^{2} + (8 y^{2} + 56 y + 63) x + y^{4} + 63 y^{3} + 651 y^{2} + 1387 y + 595,\\
   R_{\cM_2}&=x^{2} + (6 y^{2} + 48 y + 63) x + y^{4} + 63 y^{3} + 651 y^{2} + 1389 y + 603.
\end{align*}
Again we observe that -- different from weakly isomorphic \qM{}s -- there does not exist a dr-bijection between the subspace lattices of the ground spaces 
of~$\cM_1$ and~$\cM_2$.
\end{exa}

\begin{exa}\label{E-FlatsNotIso}
The \qM{}s~$\cM_1,\,\cM_2$ in \cite[Ex.~3.4]{GLJ25CF} satisfy the following.
The lattices $\cL(E_1)$ and $\cL(E_2)$ have the same number of subspaces with a given dimension and rank, and the same is true for 
$\cF(\cM_1)$ and $\cF(\cM_2)$.
In other words, there exists a dr-bijection $\alpha:\cL(E_1)\longrightarrow\cL(E_2)$ whose restriction induces a dr-bijection
between $\cF(\cM_1)$ and $\cF(\cM_2)$.
However, the two lattices of flats are not isomorphic. This follows, for instance, from the fact that in~$\cM_2$ the 
flat~$\subspace{e_3}$ is contained in 28 other flats, whereas in $\cM_1$ there is no flat contained in 28 other flats.
As a consequence, the lattices $\cF(\P\cM_1)$ and $\cF(\P\cM_2)$ are not isomorphic 
and hence $\P\cM_1\not\cong\P\cM_2$. Nonetheless, it turns out that the \qM{}s $\cM_1$ and $\cM_2$ have the same 
Whitney function, and thus so do $\P\cM_1$ and~$\P\cM_2$.
\end{exa}

\section*{Open Problems}
The discussion of the last section gives rise to two immediate questions.
First, do \qM{}s that admit a dr-bijection between their lattices of flats share the same Whitney function (as is the case in \cref{E-FlatsNotIso})? 
We do not know whether the analogous question for matroids can be answered in the affirmative.
If the latter is the case, then the same is true for \qM{}s thanks to \cref{P-MPMFlats} and \cref{T-WhitneyProj}.
Secondly, suppose~$\cM_{G_1}$ and~$\cM_{G_2}$ are weakly isomorphic representable \qM{}s. 
What can be said about the higher support distributions of the rank-metric codes $\cC_1=\rs(G_1)$ and $\cC_2=\rs(G_2)$?

More generally, it remains to be investigated whether the coboundary polynomial for \qM{}s (defined suitably) also fits into the family of polynomial invariants related by monomial substitutions, and, furthermore, what can be said about the M\"obius polynomial?

Finally, can the results on polynomial invariants be generalized to $q$-polymatroids and $\F$-linear rank-metric codes?

\bibliographystyle{abbrv}

\end{document}